\documentclass[11pt, letterpaper,leqno]{amsart}
\usepackage{amssymb}
\usepackage{amsmath}
\usepackage{amsthm}
\usepackage{graphicx}
\usepackage{enumerate}
\usepackage{enumitem}
\usepackage{xcolor}
\usepackage{hyperref}

\allowdisplaybreaks

\numberwithin{equation}{section}

\newtheorem{theo}{Theorem}[section]
\newtheorem{cor}[theo]{Corollary}
\newtheorem{lm}[theo]{Lemma}
\newtheorem{prop}[theo]{Proposition}

\DeclareMathOperator{\D}{\mathbb{D}}
\DeclareMathOperator{\N}{\mathbb{N}}
\DeclareMathOperator{\C}{\mathbb{C}}
\DeclareMathOperator{\R}{\mathbb{R}}

\theoremstyle{definition}

\newtheorem{rmk}[theo]{Remark}
\newtheorem{example}{Example}
\newtheorem*{construction}{Construction}

\title[Rectifiability and Lipschitz properties of the orbits]{Semigroups of holomorphic functions; rectifiability and Lipschitz properties of the orbits}
\author[Dimitrios Betsakos]{Dimitrios Betsakos}
\author[Konstantinos Zarvalis]{Konstantinos Zarvalis}
\thanks{K. Zarvalis is supported by Junta de Andaluc\'{i}a, grant number QUAL21 005 USE.}
\address{Department of Mathematics, Aristotle University of Thessaloniki, 54124, Thessaloniki, Greece}
\email{betsakos@math.auth.gr}
\address{Instituto de Matem\'{a}ticas de la Universidad de Sevilla, Avenida de Reina Mercedes, s/n, 41012, Seville, Spain}
\email{kzarvalis@us.es}


\date{}

\begin{document}

	\begin{abstract}
	Let $(\phi_t)$ be a continuous semigroup of holomorphic functions in the unit disk. We prove that all its orbits are rectifiable and that its forward orbits are Lipschitz curves. Moreover, we find a necessary and sufficient condition in terms of hyperbolic geometry so that a backward orbit is a Lipschitz curve. We further explore the Lipschitz condition for forward orbits lying on the unit circle and then for semigroups of holomorphic functions in general simply connected domains.
	\end{abstract}

	\maketitle
	
	\tableofcontents

\section{{\bf Introduction}}
Let $D\subsetneq\mathbb{C}$ be a simply connected domain. A family $(\phi_t)_{t\ge0}$ of holomorphic functions $\phi_t:D\to D$ is called a \textit{one-parameter continuous semigroup of holomorphic functions in} $D$ (or for short \textit{semigroup in} $D$) provided the following three conditions are satisfied:
\begin{enumerate}
	\item[(i)] $\phi_0=\textup{id}_{D}$;
	\item[(ii)] $\phi_{t+s}=\phi_t\circ\phi_s$, \text{for all }$t,s\ge0$;
	\item[(iii)] $\phi_t \xrightarrow{t\to0^+} \phi_0$, locally uniformly in $D$.
\end{enumerate}
A consequence of the definition is that every \textit{term} $\phi_t$ of the semigroup is a univalent function in $D$. If, in addition, for some $t_0>0$, $\phi_{t_0}$ is onto, then every term is necessarily an automorphism of $D$ and $(\phi_t)$ is a \textit{group}.

\medskip
One of the main features of semigroups is their close relationship with vector fields and dynamical systems (see e.g. \cite[Chapter 2]{Per}). More specifically, given a semigroup $(\phi_t)$ in $D$, there exists a unique holomorphic semi-complete vector field $G:D\to\mathbb{C}$ such that
\begin{equation}\label{generator}
	\dfrac{\partial\phi_t(\zeta)}{\partial t}=G(\phi_t(\zeta)), \quad\quad \zeta\in D, \quad t\in[0,+\infty).
\end{equation}
Therefore, a semigroup $(\phi_t)$ in $D$ can be thought of as the \textit{flow} of the vector field $G$. The converse process is also well-defined. The unique such mapping $G$ is called the \textit{infinitesimal generator} of the semigroup. As it turns out, the infinitesimal generator and its properties reveal various information about the semigroup. This will be the focal point of this present work. 

In recent literature, the usual reference domain when studying semigroups is the unit disk $\D$. We are going to follow this trend and mostly work with semigroups in $\D$. These semigroups  were introduced by Berkson and Porta \cite{berkson} in 1978. Since then, the field has enjoyed remarkable prosperity. A comprehensive overview of the main theory along with the  recent achievements may be found in the monograph \cite{book}.

Let $(\phi_t)$ be a semigroup in $\mathbb D$. 	For a point $z\in\mathbb{D}$, the curve 
	\begin{equation}\label{i2}
	\gamma_z:[0,+\infty)\to\mathbb{D},\;\;\;\; \gamma_z(t)=\phi_t(z),
\end{equation} 	
	 is called the \textit{forward orbit} of $z$. 
	 The point $z$ is its \textit{starting point}. It is proved in \cite{berkson} that, in fact, each $\gamma_z$ is a real analytic curve and by (\ref{generator}),
\begin{equation}\label{i3}
\gamma_z^\prime(t)=G(\gamma_z(t)),\;\;\;\;t\ge0, \;\;z\in\D.
\end{equation}

\medskip

Naturally, we are interested in the asymptotic behavior of the forward orbits, as $t\to+\infty$. By the continuous version of the Denjoy-Wolff Theorem (see e.g. \cite[Theorem 5.5.1]{Abate}), if the semigroup is not a group of hyperbolic rotations, then there exists a unique point $\tau\in\overline{\mathbb{D}}$ such that
	\begin{equation}\label{i4}
	\lim\limits_{t\to+\infty}\phi_t(z)=\lim_{t\to+\infty}\gamma_z(t)=\tau, \quad\text{for all }z\in\mathbb{D}.
	\end{equation} 
		The position of the Denjoy-Wolff point $\tau$ in the closure of the unit disk  leads to a first classification within the class of semigroups:
	\begin{enumerate}
		\item[(1)] if $\tau\in\mathbb{D}$, then $(\phi_t)$ is called \textit{elliptic},
		\item[(2)] whereas if $\tau\in\partial\mathbb{D}$, then $(\phi_t)$ is called \textit{non-elliptic}.
	\end{enumerate}

\medskip

If $z\in\D$, we can also think of a backward orbit of a semigroup starting from $z$ until it reaches a point on the unit circle. More precisely, let $z\in \D$. The function $\phi_t^{-1}$ is defined at $z$ for every $t$ in an interval of the form $[0,T)$.   Let $T_z$ be the supremum of all such numbers $T$.   
	The continuous curve $\tilde\gamma_z:[0,T_z)\to\mathbb{D}$ with
		\begin{equation}\label{i5}
\tilde\gamma_z(t)=\phi_t^{-1}(z)	
\end{equation} 	
is called {\it  the backward orbit} of the semigroup starting from $z$.
 Backward orbits were initially studied in \cite{CD} and \cite{ESZ}. Further results appear in \cite{Arosio}, \cite[Chapter 13]{book},  \cite{EKRS},  \cite{KZ}, \cite{illinois}.

\medskip

Various properties of the orbits have been studied; see e.g. \cite{Bet1},   \cite{BK}, \cite{BCDG2}, \cite{asympt},  \cite{Gum}, \cite{illinois}. It has been proved, in particular, that there exists a  non-elliptic semigroup in $\mathbb D$ such that each one of its forward orbits oscillates as it approaches the Denjoy-Wollf point; see \cite{Bet1}, \cite[Ch. 17]{book}, \cite{CDG}. A similar result holds for backward orbits; see \cite{Kel}. In the case of elliptic semigroups, the orbits, in general, approach the Denjoy-Wolff point $\tau$ spiralling around it. Because of these oscillating and spiralling examples, it is natural to wonder whether an orbit can actually have infinite length. The answer to this question is negative and can be inferred easily from the \textit{Hayman-Wu Theorem} for non-elliptic semigroups and from the \textit{Bishop-Jones Theorem} for elliptic semigroups. We will prove the following:

\begin{theo}\label{rectifiability}
	Let $(\phi_t)$ be a semigroup in $\D$. Every  orbit of $(\phi_t)$ (forward or backward) is rectifiable.
\end{theo} 

Our next aim is to investigate whether the (forward or backward) orbits of a semigroup are Lipschitz.  Recall that a function $\gamma:I\to\C$, where $I\subset\R$ is some interval, is called \textit{Lipschitz} if there exists some constant $C>0$, depending solely on $\gamma$, such that
$$|\gamma(t)-\gamma(s)|<C|t-s|, \quad\text{for all }t,s\in I.$$

A first observation is that, by Theorem \ref{rectifiability}, every orbit of a semigroup in $\D$ can be reparametrized using the arc-length parameter and thus it has a Lipschitz parameterization. Nevertheless, in the study of semigroups, it clearly makes more sense to deal with orbits  with the parameterization induced directly by the semigroup, and examine whether they are Lipschitz. Partial results on this subject appear in \cite[Proposition 10.1.7]{book} and \cite[Proposition 3.1]{illinois}. Our aim is to completely find out under what conditions an orbit, forward or backward, is Lipschitz or not. Our main result concerning forward orbits is the following:

\begin{theo}\label{forward}
	Let $(\phi_t)$ be a semigroup in $\D$. Every forward orbit of $(\phi_t)$ is Lipschitz.
\end{theo}

Next, in order to work with backward orbits, we are going to need an important tool of semigroup theory, the \textit{Koenigs function} (see e.g. \cite[Section 5.7]{Abate} or \cite[Chapter 9]{book}). The definition of this function depends on the type (elliptic or non-elliptic) of the semigroup. For an elliptic semigroup $(\phi_t)$ with Denjoy-Wolff point $\tau\in\D$, its Koenigs function is the (unique up to a multiplicative complex constant) conformal mapping $h:\D\to\C$ such that $h(\tau)=0$ and
\begin{equation}\label{Koenigs elliptic}
	h(\phi_t(z))=e^{-\mu t}h(z), \quad\textup{for all }t\ge0\textup{ and }z\in\D,
\end{equation}
where $\mu\in\C$, with $\textup{Re}\mu>0$. The number $\mu$ is called the \textit{spectral value} of $(\phi_t)$ and is the unique complex number satisfying $\phi_t'(\tau)=e^{-\mu t}$, for all $t\ge0$ (see \cite[Definition 5.2.10]{Abate}). When $(\phi_t)$ is non-elliptic, its Koenigs function is the (unique up to an additive complex constant) conformal mapping such that
\begin{equation}\label{Koenigs non-elliptic}
	h(\phi_t(z))=h(z)+t, \quad\textup{for all }t\ge0\textup{ and }z\in\D.
\end{equation} 
In both cases, the simply connected domain $\Omega:=h(\D)$ is called the \textit{Koenigs domain} of the semigroup.


Our main result concerning backward orbits will be stated in hyperbolic terms. Given a simply connected domain $\Omega\subsetneq\C$, we denote by $\lambda_{\Omega}$ the \textit{hyperbolic density} in $\Omega$ and by $k_\Omega$ the corresponding \textit{hyperbolic distance} (more details on hyperbolic quantities follow in Section 2).

\begin{theo}\label{backward}
	Let $(\phi_t)$ be a semigroup in $\D$ with Koenigs function $h$ and Koenigs domain $\Omega$. Let $z\in\D$. 
	\begin{enumerate}
		\item[\textup{(a)}] Suppose $(\phi_t)$ is non-elliptic. The backward orbit $\tilde{\gamma}_z$ is Lipschitz if and only if
		\begin{equation}\label{T3e1}
			\limsup\limits_{t\to T_z}\dfrac{\lambda_{\Omega}(h(z)-t)}{e^{2k_\Omega(h(z),h(z)-t)}}<+\infty.
		\end{equation}
		\item[\textup{(b)}] Suppose $(\phi_t)$ is elliptic with spectral value $\mu\in\C$, $\textup{Re}\mu>0$. The backward orbit $\tilde{\gamma}_z$ is Lipschitz if and only if
		\begin{equation}\label{T3e2}
			\limsup\limits_{t\to T_z}\dfrac{e^{\textup{Re}\mu t}\lambda_{\Omega}(e^{\mu t}h(z))}{e^{2k_\Omega(h(z),e^{\mu t}h(z))}}<+\infty.
		\end{equation}
		
	\end{enumerate}
\end{theo}
Contrary to forward orbits, we see that backward orbits are not necessarily Lipschitz. Indeed, in Section 4 we will construct a semigroup that has a non-Lipschitz backward orbit.

The condition described in Theorem \ref{backward} provides an intuitive way of understanding whether a backward orbit is Lipschitz based on the geometry of the Koenigs domain. In particular, it may lead to simple sufficient conditions.

\begin{cor}\label{regular}
	Let $(\phi_t)$ be a semigroup in $\D$. Every regular backward orbit of $(\phi_t)$ is Lipschitz.
\end{cor}

The notion of regular backward orbit will be defined in Subsection \ref{2.2}.

\begin{cor}\label{convex}
	Let $(\phi_t)$ be a semigroup in $\D$ such that its Koenigs domain $\Omega$ is convex. Every backward orbit of $(\phi_t)$ is Lipschitz.
\end{cor}
	
\medskip

By \cite{Gum}, we know that given $\zeta\in\partial\D$, the non-tangential limit $\phi_t(\zeta):=\angle\lim_{z\to\zeta}\phi_t(z)$ exists always finitely. In particular, the function
\begin{equation}\label{boundary orbit}
	[0,+\infty)\ni t\mapsto\phi_t(\zeta)
\end{equation}
is continuous. Therefore, we may consider forward orbits with starting points on the unit circle. A forward orbit emanating from a point on the boundary may lie fully on the unit circle or have an initial part on the boundary (even if it is a singleton) and then fall inside the unit disk until it reaches the Denjoy-Wolff point. The part of the forward orbit lying inside $\D$ can be considered as the union of a backward and a usual forward orbit. Thus it can be treated through the previous results. For this reason, given $\zeta\in\partial\D$ we will only deal with the part $\gamma_\zeta([0,+\infty))\cap\partial\D$ and only in the case when this intersection is larger than a singleton (if it is a singleton, then there is no sense in examining the Lipschitz condition). We will call such sets \textit{boundary orbits} for the sake of convenience. Thus, for $\zeta\in\partial\D$, we will denote by $\gamma_\zeta$ only the boundary orbit and not the full forward orbit of $\zeta$. The interested reader may refer to \cite{BG} for the first appearance of such orbits. As circular arcs, boundary orbits are clearly rectifiable. So, we will only examine the Lipschitz condition for such orbits.

Boundary orbits require a different approach than before. We are going to classify boundary orbits into three categories and treat each of them separately. We will proceed to a brief explanation of these categories and their behavior concerning the Lipschitz condition. If the boundary orbit reaches the Denjoy-Wolff point, we will call it \textit{exceptional}. So exceptional orbits may appear only in non-elliptic semigroups. On the other hand, if the boundary orbit ends at some point $\sigma$ other than the Denjoy-Wolff point and the image through the Koenigs function of $\sigma$ has positive distance from the rest of the boundary of the Koenigs domain, we will call it \textit{isolated}. These definitions are inspired from \cite[Chapter 14]{book} and \cite[Chapter 15]{book}, respectively. In Theorem \ref{thm:exc-isol}, we will prove that all exceptional orbits are Lipschitz and all isolated orbits are not Lipschitz. Finally, the third category contains all the boundary orbits that do not fall in the first two. We will call them \textit{boundary orbits of the third type}. In this case, some of them are Lipschitz and some are not. For those orbits we are going to prove Theorem \ref{thm:third} which offers a necessary and sufficient condition similar to that in Theorem \ref{backward}.

\medskip

The structure of the article is as follows: First, in Section 2 we are going to mention all the information that is deemed necessary for our proofs. We will prove Theorem \ref{rectifiability} in Section 3. Next, in Section 4 we will deal with orbits strictly contained in the unit disk proving all the relative results with regard to the Lipschitz condition. Furthermore, we will provide explicit examples demonstrating that the situation described in Theorem \ref{backward} can indeed occur. Then, in Section 5 we work with orbits lying on the unit circle. Finally, in Section 6 we are going to examine under what assumptions our results may be extended to a semigroup in any simply connected domain $D\subsetneq\C$ other than the unit disk $\D$.


\section{{\bf Preliminaries}}

\subsection{Semigroup theory}\label{2.1}
Let $(\phi_t)$ be a semigroup in $\D$ with Koenigs function $h$ and Koenigs domain $\Omega$. When $(\phi_t)$ is elliptic with spectral value $\mu$, the Koenigs domain is \textit{$\mu$-spirallike}; this means that $e^{-\mu t}\Omega\subseteq\Omega$, for all $t\ge0$. Moreover, given $z\in\mathbb{D}$, the image of its forward orbit is the half-spiral $\{e^{-\mu t}h(z):t\ge0\}$. For non-elliptic semigroups, $\Omega$ is \textit{convex in the positive direction} (also known as \textit{starlike at infinity}); this means that $\Omega+t\subseteq\Omega$, for all $t\ge0$. 
The image of the forward orbit $\gamma_z$ through $h$ is the horizontal half-line $\{h(z)+t:t\ge0\}$. 

\medskip




\medskip

 Let $G$ be the infinitesimal generator of $(\phi_t)$. There exists a useful relation between the Koenigs function and the infinitesimal generator. Indeed (see \cite[Theorem 10.1.4 and Corollary 10.1.12]{book}), for an elliptic semigroup with spectral value $\mu$ and Denjoy-Wolff point $\tau$,
\begin{equation}\label{hG elliptic}
G(z)=-\mu\frac{h(z)}{h'(z)}\;\;\;\hbox{and}\;\;\;G'(\tau)=-\mu,\;\;\;z\in\D.
\end{equation}
On the other side, when $(\phi_t)$ is non-elliptic,
\begin{equation}\label{hG nonelliptic}
G(z)=\frac{1}{h'(z)},\;\;\;z\in\D. 
\end{equation}

\medskip

\subsection{Hyperbolic geometry}\label{hyper-subsection}

Even though we primarily work with Euclidean quantities, it will turn out later on that the hyperbolic geometry can provide valuable assistance with regard to the Lipschitz condition. In this part of the present article, we will briefly mention certain hyperbolic quantities that we are going to need. For a more detailed presentation of hyperbolic geometry, we refer the interested reader to \cite[Chapter 5]{book}.

We start with the \textit{hyperbolic metric} in the unit disk $\D$ which is given through the formula
\begin{equation}\label{hyperbolicmetric}
	\lambda_{\D}(z)|dz|=\frac{1}{1-|z|^2}|dz|, \quad z\in\D.
\end{equation}
The function $\lambda_{\D}$ is called the \textit{hyperbolic density} of $\D$. In addition, the \textit{hyperbolic distance} $k_{\D}$ in the unit disk $\D$ is given by
\begin{equation}\label{hyperbolicdistance}
	k_{\D}(z,w)=\frac{1}{2}\log\frac{|1-\bar{z}w|+|z-w|}{|1-\bar{z}w|-|z-w|}, \quad z,w\in\D.
\end{equation}
If $\Omega\subsetneq\C$ is a simply connected domain and $f:\Omega\to\D$ is a corresponding Riemann map, then we may define the hyperbolic density $\lambda_{\Omega}$ and the hyperbolic distance $k_\Omega$ in $\Omega$ through the relations
\begin{equation}\label{hyperbolicinvariance}
	\lambda_{\Omega}(z)=\lambda_{\D}(f(z))|f'(z)|, \quad k_\Omega(z,w)=k_{\D}(f(z),f(w)),
\end{equation}
where $z,w\in\Omega$. This definition is independent of the choice of the Riemann mapping $f$.

The hyperbolic quantities inside a simply connected domain $\Omega$ are closely associated with the Euclidean distance from the boundary $\partial\Omega$. This correlation will be very useful in the sequel and may be demonstrated by means of the following inequalities. For the rest of the present work, given a domain $\Omega\subsetneq\C$ and $z\in\Omega$, we are going to use the notation $\delta_\Omega(z):=\textup{dist}(z,\partial\Omega)$.

\begin{lm}{\cite[Theorem 5.2.1]{book}}\label{hyperbolicmetric-Euclideandistance}
	Let $\Omega\subsetneq\C$ be a simply connected domain. Then, for every $z\in\Omega$,
	\begin{equation*}
		\frac{1}{4\delta_\Omega(z)}\le\lambda_{\Omega}(z)\le\frac{1}{\delta_\Omega(z)}.
	\end{equation*}
\end{lm}

\begin{lm}{\cite[Theorem 5.3.1, Theorem 5.3.3]{book}}\label{distancelemma}
	Let $\Omega\subsetneq\C$ be a simply connected domain. Then, for every $z,w\in\Omega$,
	\begin{equation}\label{distancelemmaequation}
		k_{\Omega}(z,w)\ge\frac{1}{4}\log\left(1+\frac{|z-w|}{\min\{\delta_\Omega(z),\delta_\Omega(w)\}}\right).
	\end{equation}
	If, in addition, $\Omega$ is convex, then the constant $\frac{1}{4}$ may be replaced by $\frac{1}{2}$.
\end{lm}

\medskip

\subsection{Backward orbits}\label{2.2}
Given a semigroup $(\phi_t)$ and a point $z\in\D$, the backward orbit $\tilde{\gamma}_z:[0,T_z)\to\D$ is called \textit{regular} if $T_z=+\infty$ and
$$\limsup\limits_{t\to+\infty}k_{\D}(\tilde{\gamma}_z(t),\tilde{\gamma}_z(t+1))<+\infty.$$
If $T_z=+\infty$ and the above upper limit  is infinite, the backward orbit is called \textit{non-regular}.

It can be readily verified that the image of the backward orbit $\tilde{\gamma}_z$ through the Koenigs function $h$ is either the spiral arc or half-spiral $\{e^{\mu t}h(z):t\in[0,T_z)\}$ when $(\phi_t)$ is elliptic with spectral value $\mu$, or the horizontal segment or half-line $\{h(z)-t:t\in[0,T_z)\}$ when $(\phi_t)$ is non-elliptic.

Contrary to forward orbits, the asymptotic behavior of backward orbits conceals some intricacies. Depending on the type of the semigroup (elliptic or non-elliptic) and the type of the backward orbit (regular or non-regular), the convergence of $\tilde{\gamma}_z$, as $t\to T_z$, varies. In fact, when $T_z<+\infty$, $\tilde{\gamma}_z$ converges to some point of $\partial\mathbb{D}\setminus\{\tau\}$, tangentially or not. When $T_z=+\infty$, combining \cite[Lemma 13.1.5, Proposition 13.1.7]{book}, we have the following:

\begin{prop}\label{backward-convergence}
Let $(\phi_t)$ be a semigroup in $\mathbb{D}$ with Denjoy-Wolff point $\tau$. Let $z\in\D$ with $T_z=+\infty$. Then:
\begin{enumerate}
\item[\textup{(i)}] If $(\phi_t)$ is elliptic (i.e. $\tau\in\mathbb{D}$), then
\begin{enumerate}
	\item[\textup{(i1)}] either $z=\tau$ and $\tilde{\gamma}_z(t)=\tau$, for all $t\ge0$,
	\item[\textup{(i2)}] or $z\ne\tau$, $\tilde{\gamma}_z$ is regular and converges to a point of the unit circle  non-tangentially,
	\item[\textup{(i3)}] or $z\ne\tau$, $\tilde{\gamma}_z$ is non-regular and converges to a point of the unit circle in any manner.
\end{enumerate}
\item[\textup{(ii)}] If $(\phi_t)$ is non-elliptic (i.e. $\tau\in\partial\mathbb{D})$, then
\begin{enumerate}
	\item[\textup{(ii1)}] either $\tilde{\gamma}_z$ is regular and converges tangentially to $\tau$,
	\item[\textup{(ii2)}] or $\tilde{\gamma}_z$ is regular and converges non-tangentially to some $\sigma\in\partial\mathbb{D}\setminus\{\tau\}$,
	\item[\textup{(ii3)}] or $\tilde{\gamma}_z$ is non-regular and converges to some $\sigma\in\partial\mathbb{D}\setminus\{\tau\}$ in any manner.
\end{enumerate}
\end{enumerate}
\end{prop}
From now on, we disregard the case when the semigroup is elliptic and the backward orbit is constant and equal to the Denjoy-Wolff point. As a matter of fact, this last proposition shows that when working in backward terms, elliptic and non-elliptic semigroups are not really different from a dynamical standpoint. 

The existence of a regular backward orbit $\tilde{\gamma}_z$ discloses helpful geometric information about the Koenigs domain $\Omega$ of a semigroup $(\phi_t)$. More specifically, we have the following implications:
\begin{enumerate}
	\item[(a)] If $(\phi_t)$ is elliptic, then there exists a maximal spirallike sector $V$ such that $\tilde{\gamma}_z([0,+\infty))\subset V\subset\Omega$.
	\item[(b)] If $(\phi_t)$ is non-elliptic and $\tilde{\gamma}_z$ converges to $\tau$, then there exists a maximal horizontal half-plane $H$ such that $\tilde{\gamma}_z([0,+\infty))\subset H\subset\Omega$.
	\item[(c)] If $(\phi_t)$ is non-elliptic and $\tilde{\gamma}_z$ converges to a point in $\partial\D\setminus\{\tau\}$, then there exists a maximal horizontal strip $S$ such that $\tilde{\gamma}_z([0,+\infty))\subset S\subset\Omega$.
\end{enumerate}
In all three cases, the notion of maximality signifies that there exists no other spirallike sector/horizontal half-plane/horizontal strip properly containing $V$/$H$/$S$ and contained inside $\Omega$.

To end the section, we make one brief mention about the ``union'' of a forward with a backward orbit. Let $z\in\D$. Then the curve $\hat{\gamma}_z:(-T_z,+\infty)\to\D$ defined through
$$\hat{\gamma}_z(t)=
\begin{cases}
	\gamma_z(t), \quad\text{for }t\ge0,\\
	\tilde{\gamma}_z(-t), \quad\text{for }t\le0,
\end{cases}
$$
is called the \textit{full orbit} of $z$ (also seen as \textit{maximal invariant curve} in literature). It can be easily seen that the full orbit of $z$ is essentially the same as the full orbit of any point on $\gamma_z([0,+\infty))$ or $\tilde{\gamma}_z([0,T_z))$.

\subsection{Finite shift}
Let $(\phi_t)$ be a non-elliptic semigroup in $\D$ with Denjoy-Wolff point $\tau\in\partial\D$. For $R>0$ consider $E(\tau,R)$ to be the \textit{horodisk} of $\D$ with center $\tau$ and radius $R>0$. In other words, $E(\tau,R)$ is a Euclidean disk of radius $\frac{R}{R+1}$ that is internally tangent to $\partial\D$ at the point $\tau$. Analytically, we may write
$$E(\tau,R)=\left\{z\in\D:\dfrac{|\tau-z|^2}{1-|z|^2}<R\right\}.$$
By means of these horodisks, we may proceed to a classification within the class of non-elliptic semigroups. Indeed, $(\phi_t)$ is said to be of \textit{finite shift} if for every $z\in\D$ there exists some $R_z>0$ such that the orbit $\gamma_z$ does not intersect the horodisk $E(\tau,R_z)$ even though it converges to $\tau$. For more background on finite shift, we refer the interested reader to \cite[Section 17.7]{book}. It can be proved that if there exists some $z\in\D$ which satisfies the above condition, then every $z\in\D$ does. If, on the contrary, for some (and hence every) $z\in\D$, the orbit $\gamma_z$ intersects all horodisks centered at $\tau$, then $(\phi_t)$ is said to be of \textit{infinite shift}.

Therefore, it is easily understood that whenever $(\phi_t)$ is of finite shift, then every orbit must reach the Denjoy-Wolff point $\tau$ travelling between $\partial\D$ and some horocycle $\partial E(\tau,R)$. As a result, every orbit ``stays very close'' to the unit circle while converging to $\tau$. For this reason, semigroups of finite shift are also called \textit{strongly tangential} in the literature.

Later on in the present article, we will prove a result (Proposition \ref{propshift}) about the Lipschitz condition with respect to the shift of a non-elliptic semigroup. In order to do so, we will utilize the ``translation'' of finite shift in the setting of the standard right half-plane $\mathbb{H}$. Consider $C:\D\to\mathbb{H}$ to be the Cayley transform with $C(z)=\frac{\tau+z}{\tau-z}$. It is quite straightforward that given $R>0$ the horodisk $E(\tau,R)$ is mapped through $C$ conformally onto the vertical half-plane $\{w\in\mathbb{H}:\textup{Re}w>\frac{1}{R}\}$. Consequently, $(\phi_t)$ is of finite shift if and only if for each $z\in\D$, there exists some $M_z>0$ such that 
$$\textup{Re}C(\phi_t(z))=\textup{Re}C(\gamma_z(t))<M_z, \quad\text{for all }t\ge0.$$

One final essential piece of information concerning semigroups of finite shift is the fact that they are necessarily \textit{parabolic of positive hyperbolic step}. This means that their Koenigs domains are contained in some horizontal half-plane, but not inside a horizontal strip; see \cite[Theorem 9.4.10]{book}. The actual definitions of parabolicity and hyperbolic step are different and require angular derivatives and hyperbolic geometry. Nevertheless, for our purposes, this equivalent counterpart suffices.

 \subsection{Distance from the boundary}\label{2.3}
 
 Let $(\phi_t)$ be a semigroup in $\D$ with Koenigs function $h$ and Koenigs domain $\Omega$. Fix $z\in\D$. In order to check whether the forward orbit $\gamma_z$ and the backward orbit $\tilde{\gamma}_z$ are Lipschitz, we will use inequalities concerning the distance of $h(z)$ from the boundary $\partial\Omega$. We will need the following classical inequality about it.
 
 \begin{lm}{\cite[Corollary 1.4]{pomm2}}\label{pommerenkelemma}
 	If $f$ maps $\D$ conformally into $\C$, then for all $z\in\D$
 	\begin{equation}\label{pommerenke}
 		\frac{1}{4}(1-|z|^2)|f'(z)|\le\delta_{f(\D)}(f(z))\le(1-|z|^2)|f'(z)|.
 	\end{equation}
 \end{lm}
 
 \medskip
 
Assume, first, that $(\phi_t)$ is non-elliptic. Recall that in this case, the Koenigs domain $\Omega$ is convex in the positive direction, $h(\gamma_z(t))=h(z)+t$ and $h(\tilde{\gamma}_z(t))=h(z)-t$. Combining everything, we see that $\delta_\Omega(h(z)+t)$ is a (not necessarily strictly) increasing function of $t\ge0$, while $\delta_\Omega(h(z)-t)$ is a decreasing one.
 
 When working with a forward orbit $\gamma_z$ of a non-elliptic semigroup $(\phi_t)$, the quantity $\delta_\Omega(h(z)+t)$ is always bounded from below by $\delta_\Omega(h(z))$. Of course, depending on the geometry of $\Omega$, it could also be bounded from above, as well. Either way, the distance from the boundary $\partial\Omega$ presents more interest when working with a backward orbit $\tilde{\gamma}_z$ due to the intricacies that may appear.
 
 First of all, if $T_z<+\infty$, the the image through $h$ of the orbit $\tilde{\gamma}_z$ converges to some point of $\partial\Omega$, as $t\to T_z$. Clearly then, $\lim_{t\to T_z}\delta_\Omega(h(z)-t)=0$.
 
 However, this is not always the case when $T_z=+\infty$. We distinguish the following cases:
 \begin{enumerate}
 	\item[(i)] When $\tilde{\gamma}_z$ is regular, as we said either there exists a maximal horizontal half-plane $H$ or a maximal horizontal strip $S$ contained inside $\Omega$ and containing $h\circ\tilde{\gamma}_z([0,+\infty))$. Then it is straightforward that $\delta_\Omega(h(z)-t)$ is bounded above by $\delta_\Omega(h(z))$, and we have either  $$\lim_{t\to+\infty}\delta_\Omega(h(z)-t)=\delta_H(h(z))>0$$ or $$\lim_{t\to+\infty}\delta_\Omega(h(z)-t)=\delta_S(h(z))>0.$$
 	\item[(ii)] When $\tilde{\gamma}_z$ is non-regular, then  $\lim_{t\to+\infty}\delta_\Omega(h(z)-t)=0$.
\end{enumerate}
\medskip

On the other hand, when $(\phi_t)$ is elliptic, the situation is not so clear. First and foremost, the image of the curve $\gamma_z$ through $h$ is no longer a half-line, but a half-spiral. Suppose that $\tau$ is the Denjoy-Wolff point and $\mu\in\C$, with $\text{Re}\mu>0$, is the spectral value of $(\phi_t)$. Then $h(\gamma_z(t))=e^{-\mu t}h(z)$. Clearly $\delta_\Omega(e^{-\mu t}h(z))$ is bounded for $t\ge0$. Moreover, we know that $h(\tau)=0$, which means that $\delta_\Omega(e^{-\mu t}h(z))$ converges to $\delta_\Omega(0)$. However, in general, we cannot make any further explicit observations on the upper and lower bounds or the monotonicity of $\delta_\Omega(e^{-\mu t}h(z))$. The same is true in the case of backward orbits. The only remark we can make is that when $T_z=+\infty$ and $\tilde{\gamma}_z$ is non-regular, we have $\lim_{t\to+\infty}\delta_\Omega(h(\tilde{\gamma}_z(t)))=\lim_{t\to+\infty}\delta_\Omega(e^{\mu t}h(z))=0$, while if it is regular $\lim_{t\to+\infty}\delta_\Omega(h(\tilde{\gamma}_z(t)))=\lim_{t\to+\infty}\delta_\Omega(e^{\mu t}h(z))=+\infty$ because of the spirallike sector contained inside $\Omega$.

 \subsection{Length of curves under conformal mapping}\label{hayman-wu}
 
 We denote by $\ell$ the length measure on the plane (namely, the one-dimensional Hausdorff measure). We will need (in Section 3) the following well-known 
 \textit{Hayman-Wu Theorem} which was first proved in \cite{HW} and later improved (in terms of the constant) in \cite{Oyma}.
 
 \begin{theo}\label{HWT}
 	Let  $f:\D\to\C$ be a conformal mapping and $L$ be a straight line. Then
 	\begin{equation}\label{HWTe}	\ell(f^{-1}(L\cap f(\D)))\le4\pi.
 	\end{equation}
 \end{theo}
 
 Bishop and Jones \cite{BJ} proved a more general theorem that involves Ahlfors regular sets.  We say that a connected set $L\subset \C$ is \textit{Ahlfors regular} if there exists a positive constant $C$ such that for every $w\in \C$ and every $r>0$,
 \begin{equation}\label{AR}
 \ell(L\cap\{z:|z-w|<r\})\le Cr.
 \end{equation}
 For more information about Ahlfors regularity, we refer to \cite[Chapter 7]{pomm2}.
 
 \begin{theo}\cite{BJ}\label{BJT}
 	Let $L\subset \C$ be a connected set. There exists a constant $C_L>0$ such that	for every conformal map $f:\D\to\C$,
 	\begin{equation}\label{BJTe}
 	\ell(f^{-1}(L\cap f(\D)))\le C_L
 	\end{equation}
 	if and only if $L$ is Ahlfors regular. 
 \end{theo}


 \section{{\bf Rectifiability of orbits}}
 We first prove an elementary lemma.
 \begin{lm}\label{spirals}
 Let $w_o\in\C$, $\alpha, \beta\in\R$. Consider the spiral curve 
 \begin{equation}\label{spiralse}
 \gamma(t)=w_o \;e^{(\alpha+i\beta) t},\;\;\;\; t\in [0,+\infty).
 \end{equation}
 Its trace $A=\{\gamma(t): t\in [0,+\infty)\}$ is Ahlfors regular.
 \end{lm}
\begin{proof}
If $\alpha=0$ or $\beta=0$, the result is trivial; so we assume that
$\alpha\neq 0$ and $\beta\neq 0$. Let $w\in \C$ and $r>0$. Set $\Delta=\{z:|z-w|<r\}$. We assume that $\Delta\cap A\neq\varnothing$ and set
$$
t_1=\inf\{t\ge 0:\gamma(t)\in \Delta\},$$
$$
t_2=\sup\{t\ge 0:\gamma(t)\in \Delta\}.
$$
 By continuity, $\gamma(t_1),\gamma(t_2)\in\overline{\Delta}$.
Note that
\begin{equation}\label{spiralsp1}
|\gamma(t)|=|w_o|e^{\alpha t},\;\;\;\;\;\;|\gamma^\prime(t)|=\sqrt{\alpha^2+\beta^2}\;|w_o|e^{\alpha t},\;\;\;\;t\ge 0.
\end{equation}

Case 1: $0\notin \Delta$\\
Using (\ref{spiralsp1}), we obtain
\begin{eqnarray}\label{spiralsp2}
\ell(A\cap\Delta)&\leq & \int_{t_1}^{t_2}|\gamma^\prime(t)|dt=\frac{\sqrt{\alpha^2+\beta^2}}{|\alpha|}|w_o||e^{\alpha t_1}-e^{\alpha t_2}| \nonumber\\
&=&\frac{\sqrt{\alpha^2+\beta^2}}{|\alpha|} ||\gamma(t_1)|-|\gamma(t_2)||\leq 
\frac{\sqrt{\alpha^2+\beta^2}}{|\alpha|}|\gamma(t_1)-\gamma(t_2)| \\ &\leq & 
\frac{2\sqrt{\alpha^2+\beta^2}}{|\alpha|}\;r.
\nonumber
\end{eqnarray}

Case 2: $0\in\Delta,\;\alpha>0$\\
In this case, $A\subset \{z:|z|\ge|w_o|\}$ and the proof is identical to that in Case 1.

Case 3:  $0\in\Delta,\;\alpha<0$\\
In this case,  $A\subset \{z:|z|\le|w_o|\}$, $0\leq t_1<+\infty$ and $t_2=+\infty$. So
\begin{eqnarray}\label{spiralsp3}
\ell(A\cap\Delta)&\leq & \int_{t_1}^{+\infty}|\gamma^\prime(t)|dt=\frac{\sqrt{\alpha^2+\beta^2}}{|\alpha|}|w_o|e^{\alpha t_1} \nonumber\\
&=&\frac{\sqrt{\alpha^2+\beta^2}}{|\alpha|} |\gamma(t_1)|=
\frac{\sqrt{\alpha^2+\beta^2}}{|\alpha|}|\gamma(t_1)-0| \\ &\leq & 
\frac{2\sqrt{\alpha^2+\beta^2}}{|\alpha|}\;r.
\nonumber
\end{eqnarray}
We conclude that $A$ is Ahlfors regular.
\end{proof} 
 
\medskip

\begin{proof}[Proof of Theorem \ref{rectifiability}]
Let  $(\phi_t)$ be a semigroup with Koenigs function $h$.
Suppose first that $(\phi_t)$ is non-elliptic. Then, by (\ref{Koenigs non-elliptic}), the image under $h$ of any orbit of  $(\phi_t)$ (forward or backward) 
lies on a straight line. Hence, by the Hayman-Wu Theorem \ref{HWT}, every orbit is rectifiable.

Next, suppose that  $(\phi_t)$ is elliptic. Then, by (\ref{Koenigs elliptic}),  the image under $h$ of any orbit of  $(\phi_t)$ lies on a curve of the form (\ref{spiralse}); for forward orbits $\alpha< 0$, while for backward orbits $\alpha> 0$. It follows from Lemma \ref{spirals} that the trace of any orbit is Ahfors regular. Subsequently, by the Bishop-Jones Theorem \ref{BJT}, every orbit is rectifiable.
\end{proof}


\section{{\bf Lipschitz property of orbits in the unit disk}}

In this section, we will provide the proofs of  Theorems \ref{forward}, \ref{backward} and of several adjoint results. The theorems are stated for all semigroups, regardless of their type. However, elliptic and non-elliptic semigroups present certain dissimilarities in terms of their Koenigs function. Since our proofs are mostly based on this important tool, we are going to treat each type of semigroups separately. Hence, almost all the proofs contain two parts, one for elliptic semigroups and one for the non-elliptic ones. 

\bigskip

\subsection{Forward orbits}
We will first focus on the forward orbits of a semigroup.

\begin{proof}[Proof of Theorem \ref{forward}]
	Fix $z\in\D$ and consider $h$ to be the Koenigs function of $(\phi_t)$, $\Omega$ to be its Koenigs domain and $G$ its infinitesimal generator. In both types of semigroups, the proof hinges upon bounding the modulus of the infinitesimal generator $G$ on the forward orbit $\gamma_z$.

	\medskip
	
	Part A: Non-elliptic semigroups:
	
	Utilising Lemma \ref{pommerenkelemma}, we have for all $t\ge0$,
	\begin{eqnarray}\label{T1p1}
		|G(\gamma_z(t))|&=&\frac{1}{|h'(\gamma_z(t))|}
	\le\frac{1-|\gamma_z(t)|^2}{\delta_\Omega(h(\gamma_z(t)))}\\
		&\le&\frac{1}{\delta_\Omega(h(z)+t)}\le\frac{1}{\delta_\Omega(h(z))}=:c_z<+\infty. \nonumber
	\end{eqnarray}
	 Let $t\ge s\ge0$. Using (\ref{generator}) and (\ref{T1p1}), we obtain
	\begin{eqnarray}
		|\gamma_z(t)-\gamma_z(s)|&=&|\phi_t(z)-\phi_s(z)|
		=\left|\int\limits_{s}^{t}\frac{\partial\phi_u(z)}{\partial u}du\right| \nonumber\\
		&=&\left|\int\limits_{s}^{t}G(\phi_u(z))du\right|
		\le\int\limits_{s}^{t}|G(\gamma_z(u))|du
		\le c_z(t-s).
	\end{eqnarray}
	Therefore, $|\gamma_z(t)-\gamma_z(s)|\le c_z(t-s)$ and the forward orbit is Lipschitz.
	
		\medskip
	
	Part B: Elliptic semigroups: 
	
	Let $\tau\in\mathbb{D}$ be the Denjoy-Wolff point of $(\phi_t)$ and $\mu$ its spectral value. In this case, for the upper bound of the infinitesimal generator $G$ along the orbit $\gamma_z$,  we use (\ref{hG elliptic}) and (\ref{pommerenke}) to conclude that
	for all $t\ge0$,
	\begin{eqnarray}
		|G(\gamma_z(t))|&=&|\mu|\frac{|h(\gamma_z(t))|}{|h'(\gamma_z(t))|}
		\le|\mu|\frac{|e^{-\mu t}h(z)|(1-|\gamma_z(t)|^2)}{\delta_\Omega(h(\gamma_z(t)))}\nonumber \\
		&\le&\frac{|\mu h(z)|}{\textup{dist}(\gamma_z([0,+\infty)),\partial\Omega)}=:C_z<+\infty.
	\end{eqnarray}
	 Therefore, continuing exactly as in the preceding proof, we get that $\gamma_z$ is Lipschitz.
\end{proof}

\begin{rmk}\label{remark}
	As we mentioned in Section 2, for non-elliptic semigroups the quantity $\delta_\Omega(h(\gamma_z(t)))$ is always bounded below by $\delta_\Omega(h(z))>0$ for all $t\ge0$, for each and every $z\in\D$. Moreover, for non-elliptic semigroups, $|\gamma_z(t)|$ tends to 1, as $t\to+\infty$, since every orbit of a non-elliptic semigroup converges to the Denjoy-Wolff point, which lies on the unit circle. Therefore, following the procedure of the last proof, we understand that $\lim_{t\to+\infty}G(\gamma_z(t))=0$, for all $z\in\D$. So, if we restrict to non-elliptic semigroups, $\lim_{t\to+\infty}G(\phi_t(z))=0$ even for orbits converging to $\tau\in\partial\D$ tangentially. This observation provides a slightly stronger result than the already known angular limit $\angle\lim_{z\to\tau}G(z)=0$ (see \cite[Corollary 10.1.2]{book}).
\end{rmk}

To end this present subsection, we provide one more result correlating the Lipschitz condition with the shift of non-elliptic semigroups.

\begin{prop}\label{propshift}
	Let $(\phi_t)$ be a parabolic semigroup of positive hyperbolic step with Denjoy-Wolff point $\tau\in\partial\D$. Consider $C:\D\to\mathbb{H}$ be the Cayley transform with $C(z)=\frac{\tau+z}{\tau-z}$. The following are equivalent:
	\begin{enumerate}
		\item[\textup{(i)}] $(\phi_t)$ is of finite shift,
		\item[\textup{(ii)}] the curve $C\circ\gamma_z$ is Lipschitz for all $z\in\D$.
	\end{enumerate}
\end{prop}
\begin{proof}
	Let $h$ be the Koenigs function of $(\phi_t)$, $\Omega$ its Koenigs domain and $G$ its infinitesimal generator. Fix $z\in\D$. Set 
	$$
	\delta(t)=C(\phi_t(z))=C(\gamma_z(t))= \frac{\tau+\gamma_z(t)}{\tau-\gamma_z(t)}, \;\;\;t\ge0.
	$$ 
	Our first aim is to evaluate the derivative of $\delta$.  
	$$\delta'(t)=\frac{\gamma_z'(t)(\tau-\gamma_z(t))+\gamma_z'(t)(\tau+\gamma_z(t))}{(\tau-\gamma_z(t))^2}=\frac{2\tau\gamma_z'(t)}{(\tau-\gamma_z(t))^2}.$$
	But we know that 
	$$
	\gamma_z'(t)=\frac{\partial\phi_t(z)}{\partial t}=G(\gamma_z(t))=\frac{1}{h'(\gamma_z(t))},
	$$
	where $h'$ is the derivative of the Koenigs function with respect to $z$. Through Lemma \ref{pommerenkelemma}, we get
	$$\frac{1-|\gamma_z(t)|^2}{4\delta_\Omega(h(\gamma_z(t)))}\le\frac{1}{|h'(\gamma_z(t))|}\le\frac{1-|\gamma_z(t)|^2}{\delta_\Omega(h(\gamma_z(t)))}.$$
	As a result, combining everything, we arrive to the conclusion
	$$\frac{1-|\gamma_z(t)|^2}{2\delta_\Omega(h(z)+t)|\tau-\gamma_z(t)|^2}\le |\delta'(t)|\le\frac{2(1-|\gamma_z(t)|^2)}{\delta_\Omega(h(z)+t)|\tau-\gamma_z(t)|^2}.$$
	Now, suppose that $(\phi_t)$ is of finite shift. Then, by definition, there exists some $R_z>0$ such that for all $t\ge0$, $$
	\frac{|\tau-\gamma_z(t)|^2}{1-|\gamma_z(t)|^2}\ge R_z.
	$$
	Consequently, for all $t\ge0$,
	$$
	|\delta'(t)|\le\frac{2}{\delta_\Omega(h(z))R_z}<+\infty.
	$$
	Since the derivative is bounded, we infer that $\delta=C\circ\gamma_z$ is Lipschitz. 
	
	Conversely, suppose that $(\phi_t)$ is of infinite shift. Then, for every $R>0$, there exists $t_R\ge0$ such that for all $t\ge t_R$, $$
	\frac{|\tau-\gamma_z(t)|^2}{1-|\gamma_z(t)|^2}<R.
	$$ 
	Therefore, for every $R>0$, there exists $t_R\ge0$ so that for all $t\ge t_R$, 
	$$
	|\delta'(t)|\ge\frac{1}{2\delta_\Omega(h(z)+t)R}.
	$$
	But $(\phi_t)$ is parabolic of positive hyperbolic step, which implies that $\Omega$ is contained in some horizontal half-plane. Hence $\delta_\Omega(h(z)+t)$ is bounded and we immediately get $\limsup_{t\to+\infty}|\delta'(t)|=+\infty$. As a result, we deduce that $\delta=C\circ\gamma_z$ is not Lipschitz.
\end{proof}

\bigskip

\subsection{Backward orbits}
We proceed to the backward orbits of semigroups. We will first need an easy lemma correlating the Lipschitz condition for a backward orbit $\tilde{\gamma}_z$ with the boundedness of $|G(\tilde{\gamma}_z(t))|$, as $t\to T_z$. In the case of forward orbits, its analogue is trivial because of (\ref{i3}). 

\begin{lm}\label{generator-lemma}
	Let $(\phi_t)$ be a semigroup in $\D$ with infinitesimal generator $G$. Then, for $z\in\D$, the backward orbit $\tilde{\gamma}_z$ is Lipschitz if and only if $\limsup_{t\to T_z}|G(\tilde{\gamma}_z(t))|<+\infty$.
\end{lm}
\begin{proof}
	Suppose first that the limsup is finite. Then there exists some $c=c(z)\in(0,+\infty)$ such that $|G(\tilde{\gamma}_z(t))|<c$, for all $t\in[0,T_z)$. Fix $0\le t_1\le t_2< T_z$. Pick an arbitraty $s\in(t_2,T_z)$ and set $s_i=s-t_i$, $i=1,2$. Then
	\begin{eqnarray*}
		|\tilde{\gamma}_z(t_2)-\tilde{\gamma}_z(t_1)|&=&|\tilde{\gamma}_z(s-s_2)-\tilde{\gamma}_z(s-s_1)|=|\phi_{s_2}(\tilde{\gamma}_z(s))-\phi_{s_1}(\tilde{\gamma}_z(s))|\\
		&=&\left|\int\limits_{s_1}^{s_2}\frac{\partial\phi_u(\tilde{\gamma}_z(s))}{\partial u}du\right|=\left|\int\limits_{s_1}^{s_2}G(\phi_u(\tilde{\gamma}_z(s)))du\right|\\
		&=&\left|\int\limits_{s_1}^{s_2}G(\tilde{\gamma}_z(s-u))du\right|=\left|\int\limits_{t_1}^{t_2}G(\tilde{\gamma}_z(t))dt\right|\\
		&\le&c(t_2-t_1),
	\end{eqnarray*}
	and thus $\tilde{\gamma}_z$ is Lipschitz.
	
	For the reverse implication, suppose that the limsup is infinite. So there exists a strictly increasing sequence $\{t_n\}\subset[0,T_z)$ with $\lim_{n\to+\infty}t_n=T_z$ and $\lim_{n\to+\infty}|G(\tilde{\gamma}_z(t_n))|=+\infty$. Without loss of generality, we may assume $\lim_{t\to T_z}\textup{Re}G(\tilde{\gamma}_z(t_n))=+\infty$ (in case $\lim_{t\to T_z}\textup{Re}G(\tilde{\gamma}_z(t_n))=-\infty$ or $\lim_{t\to T_z}\textup{Im}G(\tilde{\gamma}_z(t_n))=\pm\infty$, the proof remains almost identical). Aiming towards a contradiction, assume that $\tilde{\gamma}_z$ is Lipschitz and hence there exists some positive constant $C=C(z)$ with $$|\tilde{\gamma}_z(t_2)-\tilde{\gamma}_z(t_1)|<C(t_2-t_1), \quad\textup{for all }0\le t_1\le t_2 < T_z.$$ 
	Let $M>C+1$. Then, there is $N\in\N$ satisfying $\textup{Re}G(\tilde{\gamma}_z(t_n))>M$, for all $n\ge N$. By the continuity of $G\circ\tilde{\gamma}_z$, we can find $T>t_N$ so that $\textup{Re}G(\tilde{\gamma}_z(t))>M-1>C$, for all $t\in(t_N,T)$. Following a similar procedure as in the former case, we obtain
	\begin{eqnarray*}
		|\tilde{\gamma}_z(T)-\tilde{\gamma}_z(t_N)|^2&=&\left|\int\limits_{t_N}^{T}G(\tilde{\gamma}_z(t))dt\right|^2\\
		&\ge&\left(\int\limits_{t_N}^{T}\textup{Re}G(\tilde{\gamma}_z(t))dt\right)^2\\
		&\ge&(M-1)^2(T-t_N)^2\\
		&>&C^2(T-t_N)^2.
	\end{eqnarray*}
	But this violates the Lipschitz condition. Contradiction! Therefore, $\tilde{\gamma}_z$ is not Lipschitz.
\end{proof}

We may now provide the proof of Theorem \ref{backward}.

\begin{proof}[Proof of Theorem \ref{backward}]
	Fix $z\in\D$ and let $G$ be the infinitesimal generator of $(\phi_t)$.
	
	(a) Suppose that $(\phi_t)$ is non-elliptic. By the previous lemma, $\tilde{\gamma}_z$ is Lipschitz if and only if  $\limsup_{t\to T_z}|G(\tilde{\gamma}_z(t))|<+\infty$. So we will work with $G\circ\tilde{\gamma}_z$. By (\ref{hG nonelliptic}), we know that $G(\tilde{\gamma}_z(t))=1/h'(\tilde{\gamma}_z(t))$. Combining this with (\ref{hyperbolicinvariance}) and (\ref{hyperbolicmetric}) yields 
	\begin{equation}\label{back-1}
		|G(\tilde{\gamma}_z(t))|=\frac{\lambda_{\Omega}(h(\tilde{\gamma}_z(t)))}{\lambda_{\D}(\tilde{\gamma}_z(t))}=(1-|\tilde{\gamma}_z(t)|^2)\lambda_{\Omega}(h(z)-t),
	\end{equation}
	for all $t\in[0,T_z)$. Executing certain easy calculations, we have
	\begin{eqnarray}\label{back-2}
		\notag 1-|\tilde{\gamma}_z(t)|^2&\le&4\frac{1-|\tilde{\gamma}_z(t)|}{1+|\tilde{\gamma}_z(t)|}\\
		\notag &=&4e^{-2k_{\D}(0,\tilde{\gamma}_z(t))}\\
		\notag &\le&4e^{2k_{\D}(0,z)}e^{-2k_{\D}(z,\tilde{\gamma}_z(t))}\\
		&=&4\frac{1+|z|}{1-|z|}e^{-2k_{\Omega}(h(z),h(z)-t)},
	\end{eqnarray}
	where we made consecutive use of (\ref{hyperbolicdistance}), the triangle inequality and the conformal invariance of the hyperbolic distance. On the other hand,
	\begin{eqnarray}\label{back-3}
		\notag 1-|\tilde{\gamma}_z(t)|^2&\ge&\frac{1-|\tilde{\gamma}_z(t)|}{1+|\tilde{\gamma}_z(t)|}\\
		\notag &=&e^{-2k_{\D}(0,\tilde{\gamma}_z(t))}\\
		&\ge&\frac{1-|z|}{1+|z|}e^{-2k_{\Omega}(h(z),h(z)-t)},
	\end{eqnarray}
	where we followed a similar process as above. As a result, through (\ref{back-1}), (\ref{back-2}) and (\ref{back-3}) we obtain
	\begin{equation}\label{back-4}
		\frac{1-|z|}{1+|z|}\frac{\lambda_{\Omega}(h(z)-t)}{e^{2k_{\Omega}(h(z),h(z)-t)}}\le |G(\tilde{\gamma}_z(t))|\le 4\frac{1+|z|}{1-|z|}\frac{\lambda_{\Omega}(h(z)-t)}{e^{2k_{\Omega}(h(z),h(z)-t)}},
	\end{equation}
	for all $t\in[0,T_z)$. Therefore, $\limsup_{t\to T_z}|G(\tilde{\gamma}_z(t))|$ is finite if and only if  $$\limsup_{t\to T_z}\frac{\lambda_{\Omega}(h(z)-t)}{e^{2k_{\Omega}(h(z),h(z)-t)}}<+\infty$$  which leads to the desired result.
	
	(b) Now suppose that $(\phi_t)$ is elliptic with spectral value $\mu\in\C$, $\text{Re}\mu>0$. The proof requires similar steps. Again, for a fixed $z\in\D$, the backward orbit $\tilde{\gamma}_z$ is Lipschitz if and only if $\limsup_{t\to T_z}G(\tilde{\gamma}_z(t))<+\infty$. However, this time by (\ref{hG elliptic}), $G(\tilde{\gamma}_z(t))=-\mu h(\tilde{\gamma}_z(t))/h'(\tilde{\gamma}_z(t))$, for all $t\in[0,T_z)$. By the definition of the Koenigs function and (\ref{hyperbolicmetric}), this results in
	\begin{equation*}
		|G(\tilde{\gamma}_z(t))|=|\mu||e^{\mu t}h(z)|\frac{\lambda_{\Omega}(e^{\mu t}h(z))}{\lambda_{\D}(\tilde{\gamma}_z(t))}=|\mu h(z)|e^{\text{Re}\mu t}(1-|\tilde{\gamma}_z(t)|^2)\lambda_{\Omega}(e^{\mu t}h(z)).
	\end{equation*}
	Continuing exactly as in the non-elliptic case, we are led to the desired result.
\end{proof}

Lemma \ref{generator-lemma} gives a characterization of the Lipschitz condition for backward orbits. Other than just yielding a second characterization, the importance of Theorem \ref{backward} lies on the fact that is more easily checked through intuition and geometric consideration.  Via the last theorem, we are also able to prove certain helpful corollaries. Firstly, we will show that regularity is a sufficient condition for a backward orbit to be Lipschitz.

\begin{proof}[Proof of Corollary \ref{regular}] Fix $z\in\D$ and let $h$ be the Koenigs function of $(\phi_t)$ with $\Omega=h(\D)$. The regularity of the backward orbit dictates that $T_z=+\infty$. There are three distinct cases for the kind of $\tilde{\gamma}_z$ as evidenced by Proposition \ref{backward-convergence}. Either $(\phi_t)$ is elliptic and $\tilde{\gamma}_z(t)$ converges non-tangentially to a point on the unit circle (first kind), or $(\phi_t)$ is non-elliptic and $\tilde{\gamma}_z$ converges non-tangentially to a point on the unit circle other than the Denjoy-Wolff point (second kind), or $(\phi_t)$ is non-elliptic and $\tilde{\gamma}_z$ converges tangentially to the Denjoy-Wolff point of the semigroup (third kind).

We will first deal with the first two kinds collectively. Suppose that $\lim_{t\to +\infty}\tilde{\gamma}_z(t)=\sigma\in\partial\D$. By \cite[Proposition 12.2.4, Theorem 12.2.5]{book}, we know that $\angle\lim_{z\to\sigma}G(z)=0$. But $\tilde{\gamma}_z$ converges to $\sigma$ non-tangentially. Thus $\lim_{t\to+\infty}G(\tilde{\gamma}_z(t))=0$. At once, Lemma \ref{generator-lemma} yields that $\tilde{\gamma}_z$ is Lipschitz.

Finally, we must deal with the third kind. When passing to $\Omega$ through the Koenigs function, this signifies that $h\circ\tilde{\gamma}_z([0,+\infty))$ is contained in a maximal horizontal half-plane. The maximality of the half-plane along with the convexity of $\Omega$ in the positive direction lead to $\lim_{t\to+\infty}\delta_\Omega(h(z)-t)=:\delta_z\in(0,+\infty)$. By Theorem \ref{forward}, we are interested in the boundedness of the quantity $\lambda_{\Omega}(h(z)-t)/e^{2k_\Omega(h(z),h(z)-t)}$, as $t\to+\infty$. Through Lemma \ref{hyperbolicmetric-Euclideandistance}, we have
	\begin{equation}\label{regular1}
		\lambda_\Omega(h(z)-t)\le\frac{1}{\delta_\Omega(h(z)-t)},
	\end{equation}
	for all $t\ge0$. On the other hand, by Lemma \ref{distancelemma}, we have
	\begin{eqnarray}\label{regular2}
	\notag	k_\Omega(h(z),h(z)-t)&\ge&\frac{1}{4}\log\left(1+\frac{|h(z)-(h(z)-t)|}{\min\left\{\delta_\Omega(h(z)),\delta_\Omega(h(z)-t)\right\}}\right)\\
		&\ge&\frac{1}{4}\log\frac{t}{\delta_\Omega(h(z)-t)},
	\end{eqnarray}
	where in the last inequality we have used the fact that $\delta_\Omega(h(z)-t)$ is a decreasing quantity of $t\ge0$ due to the convexity of $\Omega$ in the positive direction. Combining relations (\ref{regular1}) and (\ref{regular2}), we obtain
	\begin{eqnarray*}
		\frac{\lambda_\Omega(h(z)-t)}{e^{2k_\Omega(h(z),h(z)-t)}}&\le&\frac{1}{\delta_\Omega(h(z)-t)}\frac{\sqrt{\delta_\Omega(h(z)-t)}}{\sqrt{t}}\\
		&=&\frac{1}{\sqrt{t\delta_\Omega(h(z)-t)}}.
	\end{eqnarray*}
	But the limit of $\delta_\Omega(h(z)-t)$, as $t\to+\infty$, exists and is greater than $0$. Hence $$\lim_{t\to+\infty}\frac{\lambda_\Omega(h(z)-t)}{e^{2k_\Omega(h(z),h(z)-t)}}=0$$ and $\tilde{\gamma}_z$ is Lipschitz by Theorem \ref{backward}.	
\end{proof}

\begin{rmk}
	In the previous corollary, the third kind of regular backward orbits required a more delicate approach due to the tangential convergence to the Denjoy-Wolff point. To prove the desired result, we used geometric techniques, both in hyperbolic and Euclidean terms. We note here that an almost identical proof works for the second kind, using a horizontal strip instead of a horizontal half-plane. Finally, similar arguments, albeit with greater difficulty and more tedious calculations, work for the first kind using a maximal spirallike sector in $\Omega$ and estimating the Euclidean distance of $\tilde{\gamma}_z$ from the boundary of the spirallike sector.
\end{rmk}

We now move on to Corollary \ref{convex} which gives a practical way of finding backward orbits that are Lipschitz, even if they are not regular. Indeed, if we construct a convex simply connected domain that is also convex in the positive direction (respectively $\mu$-spirallike), then through a Riemann mapping we may construct the associated non-elliptic (respectively elliptic with spectral value $\mu$) semigroup in $\D$ and all its backward orbits will be Lipschitz. Of course, in the elliptic case, the $\mu$-spirallike domain can only be convex if $\mu>0$ (i.e. its imaginary part is $0$).
\begin{proof}[Proof of Corollary \ref{convex}]
	Let $h$ be the Koenigs function of $(\phi_t)$ and so $\Omega=h(\D)$. Fix $z\in\D$ and consider its backward orbit $\tilde{\gamma}_z$. Once again, we will exploit Theorem \ref{backward} to produce the desired result.
	
	\medskip
	
	Part A: Non-elliptic semigroups:
	
	Following a similar procedure as in the previous corollary, using Lemma \ref{hyperbolicmetric-Euclideandistance}, we get $\lambda_{\Omega}(h(z)-t)\le\frac{1}{\delta_\Omega(h(z)-t)}$, for all $t\ge0$. On the other side, using  the convexity of $\Omega$ and Lemma \ref{distancelemma}, we get 
	$$k_\Omega(h(z),h(z)-t)\ge\frac{1}{2}\log\frac{t}{\delta_\Omega(h(z)-t)}.$$
	Combining we get
	$$\frac{\lambda_{\Omega}(h(z)-t)}{e^{2k_\Omega(h(z),h(z)-t)}}\le\dfrac{\frac{1}{\delta_\Omega(h(z)-t)}}{\frac{t}{\delta_\Omega(h(z)-t)}}=\frac{1}{t}.$$
	As a result, $$\limsup_{t\to T_z}\frac{\lambda_{\Omega}(h(z)-t)}{e^{2k_\Omega(h(z),h(z)-t)}}\le\frac{1}{T_z}\in[0,+\infty),$$ which means that $\tilde{\gamma}_z$ is Lipschitz by Theorem \ref{backward}.
	
	\medskip
	Part B: Elliptic semigroups:
	
	The proof consists of analogue arguments. Once again $\lambda_{\Omega}(e^{\mu t}h(z))\le\frac{1}{\delta_\Omega(e^{\mu t}h(z))}$. Likewise, using Lemma \ref{distancelemma} in the convex case,
	$$k_\Omega(h(z),e^{\mu t}h(z))\ge\frac{1}{2}\log\frac{|e^{\mu t}-1||h(z)|}{\delta_\Omega(e^{\mu t}h(z))}.$$
	Combining, we obtain
	$$\frac{e^{\textup{Re}\mu t}\lambda_\Omega(e^{\mu t}h(z))}{e^{2k_\Omega(h(z),e^{\mu t}h(z))}}\le\frac{\frac{e^{\textup{Re}\mu t}}{\delta_\Omega(e^{\mu t}h(z))}}{\frac{|e^{\mu t}-1||h(z)|}{\delta_\Omega(e^{\mu t}h(z))}}\le\frac{e^{\textup{Re}\mu t}}{(e^{\textup{Re}\mu t}-1)|h(z)|}.$$
	Evidently, $$\limsup_{t\to T_z}\frac{e^{\textup{Re}\mu t}\lambda_\Omega(e^{\mu t}h(z))}{e^{2k_\Omega(h(z),e^{\mu t}h(z))}}\in[1/h(z),+\infty),$$ regardless of if $T_z$ is finite or infinite. In view of Theorem \ref{backward}, we get the desired result.
\end{proof}

\begin{rmk}\label{bi-Lipschitz}
	A natural question is whether an orbit, forward or backward, of a semigroup can be bi-Lipschitz. For this to be true, we need the modulus of the infinitesimal generator to also be bounded from below. By Remark \ref{remark}, we already know that for any forward orbit $\gamma_z$, we have $\lim_{t\to+\infty}G(\gamma_z(t))=0$. So following similar methods as before, there can be no $c>0$ such that $|\gamma_z(t_2)-\gamma_z(t_1)|\ge c(t_2-t_1)$, for all $t_2\ge t_1\ge0$. Hence, no forward orbit can be bi-Lipschitz. A similar argument proves that regular backward orbits are not bi-Lipschitz either. On the other hand, if $\tilde{\gamma}_z$ is not regular (notice that not regular is not the same as non-regular), then it is entirely possible that $T_z<+\infty$ and $\lim_{t\to T_z}|G(\tilde{\gamma}_z(t))|>0$. But we know that $G$ is non-vanishing in the unit disk and therefore there exist certain backward orbits that are indeed bi-Lipschitz.
\end{rmk}

\medskip
In order for Theorem \ref{backward} to be of value, we need to prove that both  cases described in the theorem can actually occur for non-regular backward orbits. More specifically, we need to certify that there exist non-regular backward orbits which are Lipschitz and others that are not. So, we will give certain examples.

\begin{example}
	We start with an example where Theorem \ref{backward} can be used to show that a backward orbit is not Lipschitz. The example is inspired by \cite[Example 12.4.3]{book}. For $x,y\in\R$, set $L[x,y]:=\{w\in\C:\textup{Re}w\le x,\;\textup{Im}w=y\}$, a horizontal half-line stretching to infinity in the negative direction. Denote
	$$\Omega:=\{w\in\C:|\textup{Im}w|<2\}\setminus\bigcup\limits_{n=1}^{+\infty}\left(L\left[-2^n,\frac{1}{n}\right]\cup L\left[-2^n,-\frac{1}{n}\right]\right).$$
	\begin{figure}
		\centering
		\includegraphics[scale=0.45]{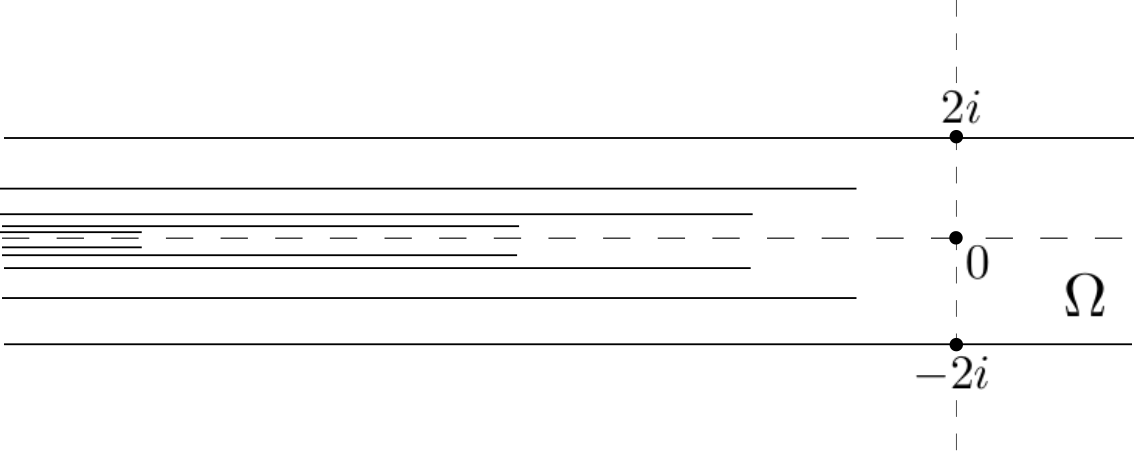}
		\caption{The Koenigs domain in Example 1}
		\label{fig:example1}
	\end{figure}
	It is readily verified that $\Omega$ is a convex in the positive direction simply connected domain (see Figure \ref{fig:example1}). Then, given a Riemann mapping $h:\D\to\Omega$ fixing the origin, the relation $\phi_t(z)=h^{-1}(h(z)+t)$, $z\in\D$, $t\ge0$, gives birth to a non-elliptic semigroup $(\phi_t)$ with Koenigs function $h$ and Koenigs domain $\Omega$. Clearly $T_0=+\infty$ and $h(\tilde{\gamma}_0(t))=-t$, for all $t\ge0$. In addition, $\tilde{\gamma}_0$ is non-regular since there is no horizontal strip or horizontal half-plane inside $\Omega$ containing $h(\tilde{\gamma}_0([0,+\infty)))$. Notice that $\delta_\Omega(h(0)-t)=\delta_\Omega(-t)\le\frac{1}{\lfloor t\rfloor}$, where $\lfloor.\rfloor$ denotes the floor function. Next, Lemma \ref{hyperbolicmetric-Euclideandistance} leads to 
	\begin{equation}\label{ex1.1}
		\lambda_{\Omega}(h(0)-t)=\lambda_\Omega(-t)\ge\frac{\lfloor t\rfloor}{4}.
	\end{equation}
	To continue with, we are going to estimate the corresponding hyperbolic distance $k_\Omega(h(0),h(0)-t)=k_\Omega(0,-t)$. Given $t\ge1$, the horizontal half-strip $\Sigma_t=\{w\in\C:\textup{Re}w>-2^{\lfloor t\rfloor},\; |\textup{Im}w|<1/\lfloor t\rfloor\}$ is contained inside $\Omega$, while also containing the rectilinear segment $[-t,0]$. By the domain monotonicity property of the hyperbolic distance, we get $k_\Omega(0,-t)\le k_{\Sigma_t}(0,-t)$, so we will work with the latter quantity. The conformal mapping 
	$$g(w)=\sin\left(\frac{i\pi\lfloor t \rfloor}{2}(w+2^{\lfloor t \rfloor})\right)$$ maps $\Sigma_t$ conformally onto the upper half-plane $\mathbb{H}$. By the conformal invariance of the hyperbolic distance, we have
	\begin{eqnarray}\label{ex1.2}
	\notag	k_\Omega(0,-t)&\le&k_{\Sigma_t}(0,-t)\\
	\notag	&=&k_{\mathbb{H}}\left(\sin\left(\frac{i\pi\lfloor t\rfloor2^{\lfloor t \rfloor}}{2}\right),\sin\left(\frac{i\pi\lfloor t\rfloor(2^{\lfloor t\rfloor}-t)}{2}\right)\right)\\
		&=&k_{\mathbb{H}}\left(i\sinh\left(\frac{\pi\lfloor t\rfloor2^{\lfloor t \rfloor}}{2}\right),i\sinh\left(\frac{\pi\lfloor t\rfloor(2^{\lfloor t\rfloor}-t)}{2}\right)\right).
	\end{eqnarray}
	Since lying on an axis of symmetry,  the segment $[-t,0]$ is a geodesic arc of $\Sigma_t$. But conformal mappings map geodesics to geodesics. Therefore, $g([0,-t])$ is still a geodesic arc of the upper half-plane. Furthermore, the hyperbolic density of $\mathbb{H}$ is given by $\lambda_{\mathbb{H}}(z)=\frac{1}{2\textup{Im}z}$. Setting
	$$
	a_t=\sinh\left(\frac{\pi\lfloor t\rfloor2^{\lfloor t \rfloor}}{2}\right)\;\;\;\; \hbox{and}\;\;\;\;b_t=\sinh\left(\frac{\pi\lfloor t\rfloor(2^{\lfloor t\rfloor}-t)}{2}\right),
	$$
	we infer that
	\begin{eqnarray*}
		k_{\mathbb{H}}\left(i\sinh\left(\frac{\pi\lfloor t\rfloor2^{\lfloor t \rfloor}}{2}\right),i\sinh\left(\frac{\pi\lfloor t\rfloor(2^{\lfloor t\rfloor}-t)}{2}\right)\right)&=&\left|\int\limits_{a_t}^{b_t}\frac{ds}{2s}\right|\\
		&=&\frac{1}{2}\log\frac{2^{\lfloor t\rfloor}}{2^{\lfloor t\rfloor}-t}.
	\end{eqnarray*}
	which through (\ref{ex1.2}) yields $k_\Omega(0,-t)\le\frac{1}{2}\log\frac{2^{\lfloor t\rfloor}}{2^{\lfloor t\rfloor}-t}$. Combining with (\ref{ex1.1}), we find
	$$\limsup\limits_{t\to+\infty}\frac{\lambda_\Omega(-t)}{e^{2k_\Omega(0,-t)}}\ge\limsup\limits_{t\to+\infty}\frac{\lfloor t\rfloor(2^{\lfloor t\rfloor}-t)}{4\cdot2^{\lfloor t\rfloor}}=+\infty.$$
	In view of Theorem \ref{backward}, we see that $\tilde{\gamma}_0$ is not Lipschitz.
\end{example}

Next, we want an example of a Lipschitz non-regular backward orbit. Towards this goal we will prove one more proposition that stems from Theorem \ref{backward}, but applies only to non-regular backward orbits of non-elliptic semigroups.

\begin{prop}\label{Lipschitz-Euclidean distance}
	Let $(\phi_t)$ be a non-elliptic semigroup in $\D$ with Koenigs function $h$ and Koenigs domain $\Omega$. Let $z\in\D$ and suppose that the backward orbit $\tilde{\gamma}_z$ is non-regular (and hence $T_z=+\infty$). If $$\liminf_{t\to+\infty}(t\cdot\delta_\Omega(h(z)-t))>0,$$ then $\tilde{\gamma}_z$ is Lipschitz.
\end{prop}
\begin{proof}
	Relying on similar arguments as in our previous results, we may see that through Lemma \ref{hyperbolicmetric-Euclideandistance}
	$$\lambda_\Omega(h(z)-t)\le\frac{1}{\delta_\Omega(h(z)-t)},$$
	while through Lemma \ref{distancelemma} we get
	\begin{eqnarray*}
	k_\Omega(h(z),h(z)-t)&\ge&\frac{1}{4}\log\left(1+\frac{t}{\min\{\delta_\Omega(h(z)),\delta_\Omega(h(z)-t)\}}\right)\\
	&\ge&\frac{1}{4}\log\frac{t}{\delta_\Omega(h(z)-t)},
	\end{eqnarray*}
	for all $t\ge0$. Blending the two relations, we obtain
	\begin{eqnarray*}
		\frac{\lambda_{\Omega}(h(z)-t)}{e^{2k_\Omega(h(z),h(z)-t)}}&\le&\frac{\frac{1}{\delta_\Omega(h(z)-t)}}{\frac{\sqrt{t}}{\sqrt{\delta_\Omega(h(z)-t)}}}\\
		&=&\frac{1}{\sqrt{t\cdot\delta_\Omega(h(z)-t)}},
	\end{eqnarray*}
	for all $t\ge0$. Therefore, our hypothesis implies $$\limsup_{t\to+\infty}\frac{\lambda_{\Omega}(h(z)-t)}{e^{2k_\Omega(h(z),h(z)-t)}}<+\infty.$$ Finally, Theorem \ref{backward} yields the Lipschitz condition for $\tilde{\gamma}_z$.
\end{proof}

This last proposition is actually very practical since it allows us to easily construct non-elliptic semigroups (or rather their associated Koenigs domains) having a non-regular Lipschitz backward orbit. Its advantage compared to Theorem \ref{backward} is that it is stated in Euclidean terms. Nevertheless, it does not provide a characterization. We will demonstrate its practicality via the next examples.

\begin{example}
	Denote by $\mathbb{H}$ the right half-plane and consider 
	$$\Omega=\overline{\mathbb{H}-1}\cup\left\{x+iy:x<-1, |y|<\frac{1}{\log|x|}\right\}.$$
	Clearly $\Omega$ is a simply connected and convex in the positive direction domain. Let $h:\mathbb{D}\to\Omega$ be a Riemann mapping fixing the origin. Then, as in our previous example, we may construct the associated non-elliptic semigroup $(\phi_t)$ through the formula $\phi_t(z)=h^{-1}(h(z)+t)$, $z\in\D$, $t\ge0$. The whole half-line $\{-t:t\ge0\}$ is contained inside $\Omega$, while there exists no horizontal strip contained inside $\Omega$ and containing this half-line. As a result, the curve $\tilde{\gamma}_0(t)=h^{-1}(-t)$ is a non-regular backward orbit for $(\phi_t)$ converging to some super-repelling fixed point of the semigroup. It is easy to see that $$
	\liminf_{t\to+\infty}(t\cdot\delta_\Omega(h(\tilde{\gamma}_0(t))))=\liminf_{t\to+\infty}(t\cdot\delta_\Omega(-t))>0.
	$$ 
	Therefore, Proposition \ref{Lipschitz-Euclidean distance} implies that $\tilde{\gamma}_0$ is Lipschitz.
\end{example}

\begin{example}
Following a thought process  similar to the previous example, consider
$$\Omega=\overline{\mathbb{H}-1}\cup\left\{x+iy:x<-1, -\frac{1}{\log|x|}-1<y<\frac{1}{\log|x|}\right\}.$$
As before, there exists a non-elliptic semigroup $(\phi_t)$ with Koenigs function $h$ such that $h(\D)=\Omega$. Clearly $\Omega$ contains the horizontal strip $S=\{z:-1<\textup{Im}z<0\}$. Therefore the backward orbit $\tilde{\gamma}_0$ with $\tilde{\gamma}_0(t)=h^{-1}(h(0)-t)=h^{-1}(-t)$ lying on $\partial S$ converges to a repelling point of $(\phi_t)$ and is non-regular. As previously, $\liminf_{t\to+\infty}(t\cdot \delta_\Omega(h(\tilde{\gamma}_0(t))))>0$ and thus $\tilde{\gamma}_0$ is Lipschitz.
\end{example}

Finally, combining all our preceding work, we may prove a relevant result for full orbits of semigroups of holomorphic functions.

\begin{cor}
	Let $(\phi_t)$ be a semigroup in $\D$ and $z\in\D$. The following are equivalent:
	\begin{enumerate}
		\item[\textup{(i)}] The full orbit $\hat{\gamma}_z:(-T_z,+\infty):\D$ is Lipschitz.
		\item[\textup{(ii)}] The backward orbit $\tilde{\gamma}_z$ is Lipschitz.
	\end{enumerate}
\end{cor}
\begin{proof}
	Recall that by definition
	$$\hat{\gamma}_z(t)=\begin{cases}
		\tilde{\gamma}_z(-t), \quad\text{when }t\in(-T_z,0], \\
		\gamma_z(t), \quad\text{when }t\in[0,+\infty).
	\end{cases}$$
Trivially (i) implies (ii). If (ii) holds, there exists some positive constant $c_1$ such that
$$|\hat{\gamma}_z(t)-\hat{\gamma}_z(s)|=|\tilde{\gamma}_z(-t)-\tilde{\gamma}_z(-s))|\le c_1 |t-s|, \quad\text{for all }t,s\in(-T_z,0].$$
We have  showed that every forward orbit of every semigroup is Lipschitz. So there exists a second positive constant $c_2>0$ so that
$$|\hat{\gamma}_z(t)-\hat{\gamma}_z(s)|=|\gamma_z(t)-\gamma_z(s)|\le c_2|t-s|, \quad\text{for all }t,s\in[0,+\infty).$$
Lastly, let $-T_z<t<0<s<+\infty$. Then
\begin{eqnarray*}
	|\hat{\gamma}_z(t)-\hat{\gamma}_z(s)|&\le&|\hat{\gamma}_z(t)-\hat{\gamma}_z(0)|+|\hat{\gamma}_z(0)-\hat{\gamma}_z(s)|\\
	&\le&c_1(-t)+c_2s\\
	&\le&\max\{c_1,c_2\}|t-s|,
\end{eqnarray*}
which leads to the desired implication.
\end{proof}

\bigskip

\section{\textbf{Lipschitz property of orbits on the unit circle}}
We move on to boundary orbits starting with some brief exposition which is mostly taken from \cite[Chapter 14]{book}. Let $(\phi_t)$ be a semigroup in $\D$. Clearly, for $z\in\D$, the trace of the orbit $\gamma_z$ is completely contained inside $\D$. Therefore, in this subsection, we only deal with orbits $\gamma_\zeta$ starting from a point $\zeta\in\partial\D$. Of course, the forward orbit of $\zeta$ either lies completely on the unit circle or has an initial part on the unit circle (even if it is just the singleton $\{\zeta\}$) and then falls into the unit disk until it reaches the Denjoy-Wolff point. This second part has already been dealt with in the previous section. So, as we mentioned in the Introduction, we will devote our study exclusively on the part $\gamma_\zeta([0,+\infty))\cap\partial\D$. It can be actually seen that this intersection is connected and is equal to $\gamma_\zeta([0,S_\zeta])$,  where $S_\zeta\in[0,+\infty]$ is called the \textit{life-time} of $\zeta$ and is given through the relation
\begin{equation}\label{life-time}
	S_\zeta:=\sup\{t\in[0,+\infty):\phi_t(\zeta)\in\partial\D\}
\end{equation}
(in case $S_\zeta=+\infty$ the intersection is the whole $\gamma_\zeta([0,+\infty))$). For the sake of convenience, we will say that the set $\gamma_\zeta([0,S_\zeta])$ (or $\gamma_\zeta([0,+\infty))$ whenever $S_\zeta=+\infty$) is the \textit{boundary orbit} of $\zeta$ and we will still denote it by $\gamma_\zeta$. It is obvious that when $S_\zeta=0$, then $\gamma_\zeta\equiv\{\zeta\}$ and there is no point in examining the Lipschitz condition. So, we will only work with $\zeta\in\partial\D$ such that $S_\zeta>0$. Any such point is called a \textit{contact point} for $(\phi_t)$. Evidently, every boundary fixed point of the semigroup is at once a contact point. In particular, if $\sigma\in\partial\D$ is a fixed point of $(\phi_t)$, then $S_\sigma=+\infty$ and $\gamma_\sigma(t)=\sigma$, for all $t\ge0$. So, again, $\gamma_\sigma\equiv\{\sigma\}$ and $\gamma_\sigma$ is easily seen to be Lipschitz. We disregard this case due to its simplicity.

An open arc (i.e. an arc not containing its endpoints) $A\subset\partial\D$ is called a \textit{contact arc} if every $\zeta\in A$ is a contact point of $(\phi_t)$. If there exists no other contact arc $B$ with $A\subsetneq B$, then $A$ is said to be a \textit{maximal} contact arc. Furthermore, if the one endpoint of a (maximal) contact arc $A$ is the Denjoy-Wolff point of the semigroup, then $A$ is known as an \textit{exceptional} (maximal) contact arc.

In order to state our main results concerning boundary orbits, we first have to proceed to a classification. We will distinguish boundary orbits into three main types, depending on their image through the Koenigs function and on whether the boundary orbit reaches the Denjoy-Wolff point or not.

	(i) Firstly, we will say that a boundary orbit $\gamma_\zeta$ is \textit{exceptional} if $S_\zeta=+\infty$ and $\zeta$ is not a fixed point of $(\phi_t)$. Clearly this signifies that $(\phi_t)$ is necessarily non-elliptic and that the arc $\gamma_\zeta([0,+\infty))$ has the Denjoy-Wolff point of $(\phi_t)$ as an endpoint.
	
	(ii) For the second type we will slightly abuse the definitions introduced in \cite[Chapter 15]{book}. Suppose that $\gamma_\zeta$ is not exceptional. Hence $S_\zeta<+\infty$. Then the image of $\gamma_\zeta$ through the Koenigs function $h$ is either the horizontal rectilinear segment $\{h(\zeta)+s:s\in[0,S_\zeta]\}$ (non-elliptic case) or the spiral segment $\{e^{-\mu s}h(\zeta):s\in[0,S_\zeta]\}$ (elliptic case with spectral value $\mu$). Then the whole horizontal half-line $L_\zeta:=\{h(\zeta)+s:s\in(-\infty,S_\zeta]\}$ or the whole half-spiral $\Lambda_\zeta:=\{e^{-\mu s}h(\zeta):s\in(-\infty,S_\zeta]\}$ is contained in $\C\setminus\Omega$. For the non-elliptic case, suppose as well that there exists $\epsilon>0$ such that $D(h(\zeta)+S_\zeta,\epsilon)\cap\partial\Omega\subset L_\zeta$, where $D(w,r)$ denotes the Euclidean disk of center $w\in\C$ and radius $r>0$. Likewise, in the elliptic case, suppose that there exists $\epsilon>0$ such that $D(e^{-\mu S_\zeta}h(\zeta),\epsilon)\cap\partial\Omega\subset \Lambda_\zeta$. In such cases, we will say that the boundary orbit $\gamma_\zeta$ is \textit{isolated}. For the sake of convenience, we might refer to $L_\zeta$ as a \textit{radial slit} and to $\Lambda_\zeta$ as a \textit{spiral slit}, while the points $h(\zeta)+S_\zeta$ and $e^{-\mu S_\zeta}h(\zeta)$ will be their \textit{tip-points}, respectively.
	
	(iii) If none of the above hold, we will say that $\gamma_\zeta$ is a \textit{boundary orbit of the third type}. 
	\begin{figure}
		\centering
		\includegraphics[scale=0.35]{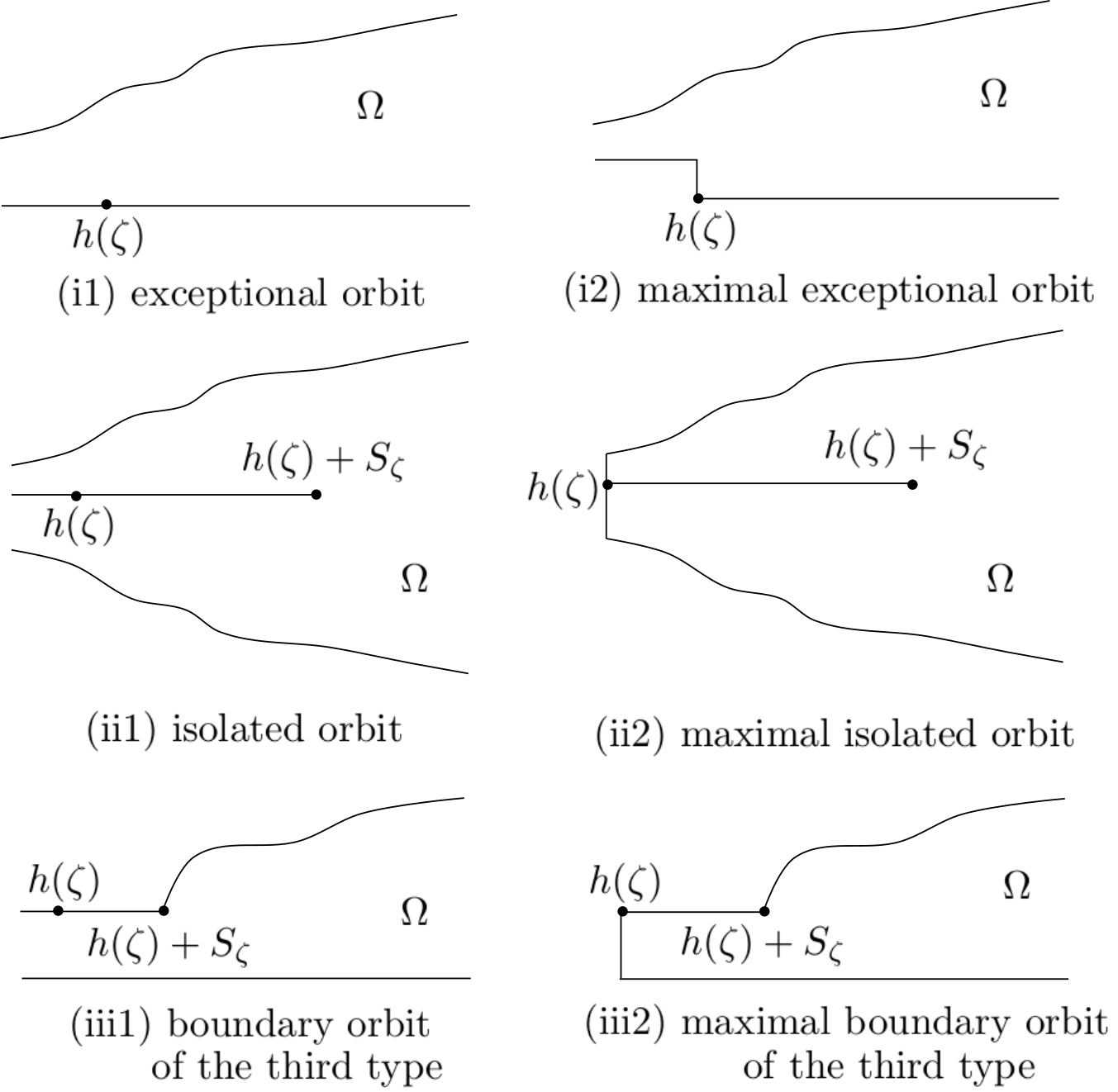}
		\caption{Images of boundary orbits through the Koenigs function in non-elliptic semigroups}
		\label{fig:boundary orbits}
	\end{figure}

	For an illustration of the preceding classification in the case of non-elliptic semigroups see Figure \ref{fig:boundary orbits}.

\medskip

Given a boundary orbit $\gamma_\zeta$, where $\zeta$ is a contact point but not a boundary fixed point, its trace on the unit circle is always a circular arc. Whenever $S_\zeta<+\infty$, then this trace is the arc with endpoints $\zeta$ and $\gamma_\zeta(S_\zeta)$ including these endpoints. On the other hand, whenever $S_\zeta=+\infty$, then the trace is the arc with endpoints $\zeta$ and the Denjoy-Wolff point, but including only $\zeta$. In any case, the trace without its endpoints is therefore a contact arc. By \cite[Theorem 14.2.10]{book}, we know that both the Koenigs function $h$ and the infinitesimal generator $G$ of a semigroup $(\phi_t)$ may be extended holomorphically through contact arcs. Therefore, if $(\phi_t)$ is non-elliptic, the equalities
\begin{equation}\label{extension}
	\frac{\partial\gamma_\zeta(t)}{\partial t}=G(\gamma_\zeta(t)) \text{ and } G(\gamma_\zeta(t))=\frac{1}{h'(\gamma_\zeta(t))},
\end{equation}
continue to hold for all $t\in(0,S_\zeta)$. In case $(\phi_t)$ is elliptic with spectral value $\mu\in\C$, $\textup{Re}\mu>0$, the second equality transforms to
\begin{equation}\label{extension2}
	G(\gamma_\zeta(t))=-\mu\frac{h(\gamma_\zeta(t))}{h'(\gamma_\zeta(t))},
\end{equation}
for all $t\in(0,S_\zeta)$.

Having provided all the necessary exposition, we are now ready to state and prove our main result concerning boundary orbits:
\begin{theo}\label{thm:exc-isol}
	Let $(\phi_t)$ be a semigroup in $\D$ and suppose that $\zeta\in\partial\D$ is not a fixed point of $(\phi_t)$ satisfying $S_\zeta>0$.
	\begin{enumerate}
		\item[\textup{(a)}] If $\gamma_\zeta$ is an exceptional orbit, then it is Lipschitz.
		\item[\textup{(b)}] If $\gamma_\zeta$ is an isolated orbit, then it is not Lipschitz.
	\end{enumerate}
\end{theo}
\begin{proof}
	As in the case of forward orbits inside the unit disk, the objective is to check whether the infinitesimal generator $G$ of $(\phi_t)$ is bounded or not along the boundary orbit. Of course, in any relatively compact subset of $\gamma_\zeta([0,S_\zeta])$ (or $\gamma_\zeta([0,+\infty))$ when $\gamma_\zeta$ is exceptional), $G$ remains bounded since it can be holomorphically extended through it. So, we only need to deal with the upper limits
	$$\limsup\limits_{t\to0}|G(\gamma_\zeta(t))| \;\;\textup{ and } \;\; \limsup\limits_{t\to S_\zeta}|G(\gamma_\zeta(t))|.$$ 
	(a) First, suppose that $\gamma_\zeta$ is exceptional and thus $S_\zeta=+\infty$. Without loss of generality, suppose that $\gamma_\zeta((0,+\infty))$ is actually a maximal contact arc. This is the hardest case. Indeed, if $\gamma_\zeta((0,+\infty))$ is not maximal, then by definition it is contained inside another contact arc. Evidently, $\zeta$ will be an interior point of that arc (i.e. not an endpoint) and $G(\zeta)$ will be well-defined. So the upper limit as $t\to0$ would be trivially finite and equal to $G(\zeta)$. Thus, assuming maximality, we know that $\angle\lim_{z\to\zeta}G(z)=:l_0$ exists finitely; cf. \cite[Theorem 14.3.1(1)]{book}. By the well-known \textit{Berkson-Porta Formula}, there exists a non-vanishing holomorphic function $p:\D\to\overline{\mathbb{H}}$, where $\mathbb{H}$ is the right half-plane, such that
	$$G(z)=(z-\tau)(\bar{\tau}z-1)p(z),$$
	where $\tau$ is the Denjoy-Wolff point of $(\phi_t)$. Clearly, this signifies that the limit $\angle\lim_{z\to\zeta}p(z)$ also exists finitely and that $p$ may be extended through $\gamma_\zeta((0,+\infty))$. In addition, it can be proved (see e.g. the proof in \cite[Theorem 14.3.1]{book}) that $\textup{Re}p(\gamma_\zeta(t))=0$, for all $t\in(0,+\infty)$. As a result, applying \cite[Proposition 2.4.2]{book}, we get
	$$\angle\lim\limits_{z\to\zeta}p(z)=\lim\limits_{t\to0}p(\gamma_\zeta(t)).$$
	Through the Berkson-Porta Formula, the same equality of limits holds for $G$ as well. Therefore, we get that the limit $\lim_{t\to0}G(\gamma_\zeta(t))$ actually exists finitely. Recall that due to the fact that $\gamma_\zeta$ is exceptional, the second endpoint of $\gamma_\zeta((0,+\infty))$ is the Denjoy-Wolff point $\tau$ of the semigroup. As we have mentioned in Remark \ref{remark}, it is known that $\angle\lim_{z\to\tau}G(z)=0$. Following exactly the same steps, we can prove that
	$$\angle\lim\limits_{z\to\tau}G(z)=\lim\limits_{t\to+\infty}G(\gamma_\zeta(t)).$$
	Hence both the limit $\lim_{t\to0}|G(\gamma_\zeta(t))|$ and the limit $\lim_{t\to+\infty}|G(\gamma_\zeta(t))|=\lim_{t\to S_\zeta}|G(\gamma_\zeta(t))|$ exist finitely. Consequently, relation (\ref{extension}) shows that $\gamma_\zeta$ is Lipschitz.
	
	(b) Next, suppose that $\gamma_\zeta$ is isolated. Ergo $S_\zeta<+\infty$. An identical proof as in the previous case shows that the limit $\lim_{t\to0}|G(\gamma_\zeta(t))|$ exists finitely regardless of whether $\gamma_\zeta((0,S_\zeta))$ is a maximal contact arc or not. We are left with evaluating the respective limit as $t\to S_\zeta$. Let $h$ be the Koenigs function of $(\phi_t)$. As we stated previously, the fact that $\gamma_\zeta$ is isolated means that $h\circ\gamma_\zeta([0,S_\zeta])$ is part of a radial slit or part of a spiral slit and that $h(\gamma_\zeta(S_\zeta))$ is its tip-point. In such cases, \cite[Theorem 15.2.6]{book} certifies that $\angle\lim_{z\to\gamma_\zeta(S_\zeta)}G(z)=\infty$. However, the same procedure as before yields
	$$\angle\lim\limits_{z\to\gamma_\zeta(S_\zeta)}G(z)=\lim\limits_{t\to S_\zeta}G(\gamma_\zeta(t)).$$
	As a result $\lim_{t\to S_\zeta}|G(\gamma_\zeta(t))|=\infty$ and (\ref{extension}) proves that $\gamma_\zeta$ is not Lipschitz.
\end{proof}

\medskip

Finally, we need to examine the Lipschitz condition for boundary orbits of the third type. Contrary to exceptional and isolated orbits, not all boundary orbits conform to the same behavior. Indeed, it will be revealed that some of them are Lipschitz and some of them are not. As with the backward orbits, our aim is to discover a necessary and sufficient condition that will help us verify the Lipschitz condition for boundary orbits of the third type. Let $\gamma_\zeta$ be such an orbit. Arguing exactly as in the proof of Theorem \ref{thm:exc-isol}, we see that $\lim_{t\to0}|G(\gamma_\zeta(t))|$ once again exists finitely. Keeping in mind that $\gamma_\zeta((0,S_\zeta))$ is still a contact arc, relation (\ref{extension}) provides a first necessary and sufficient condition: a boundary orbit $\gamma_\zeta$ of the third type is Lipschitz if and only if $\limsup_{t\to S_\zeta}|G(\gamma_\zeta(t))|$ is finite. In reality, the actual limit exists and there is no need for the upper limit. Nevertheless, the limit exists either finitely or infinitely. Both situations may arise. 

As mentioned before, the set $\gamma_\zeta((0,S_\zeta))$ is a contact arc and both the Koenigs function $h$ and the infinitesimal generator $G$ may be extended holomorphically through it. However, depending on the shape of the Koenigs domain $\Omega$, it is possible that the extended $h$ does not maintain its univalence. We will tweak this extension a little bit in order to reach the desired conclusion.

\begin{construction}
	Let $(\phi_t)$ be a non-elliptic semigroup in $\D$ with Denjoy-Wolff point $\tau\in\partial\D$, Koenigs function $h$ and Koenigs domain $\Omega=h(\D)$. Fix $\zeta\in\partial\D$ such that $S_\zeta>0$, $\zeta$ is not a fixed point of $(\phi_t)$ and the boundary orbit $\gamma_\zeta([0,S_\zeta])$ is of the third type. Since $(\phi_t)$ is non-elliptic, the image of $\gamma_\zeta$ through $h$ is just the rectilinear segment $\{h(\zeta)+t:t\in[0,S_\zeta]\}$ (recall that $h(\zeta)=\angle\lim_{z\to\zeta}h(z)$ exists in $\C$). Suppose that there exist two sequences ${t_n}$ and ${s_n}$ with $\lim_{n\to+\infty}t_n=\lim_{n\to+\infty}s_n=0$ and two sequences ${p_n}$ and ${q_n}$ with $\lim_{n\to+\infty}p_n=\lim_{n\to+\infty}q_n=S_\zeta$ such that
	$$L_n^+=\{w\in\C:\textup{Re}h(\zeta)\le \textup{Rew}\le\textup{Re}h(\zeta)+p_n,\textup{Im}w=\textup{Im}h(\zeta)+t_n\}\subset\C\setminus\Omega$$
	and
	$$L_n^-=\{w\in\C:\textup{Re}h(\zeta)\le\textup{Re}w\le\textup{Re}h(\zeta)+q_n,\textup{Im}w=\textup{Im}h(\zeta)-s_n\}\subset\C\setminus\Omega,$$
	for all $n\in\N$. In simpler words, there exist components of $\partial\Omega$ that accumulate onto $h\circ\gamma_\zeta([0,S_\zeta])$ both from above and from below. But if this were the case, then the whole rectilinear segment $h\circ\gamma_\zeta([0,S_\zeta])$ would belong in the impression of the same prime end of $\Omega$. Therefore, the preimage of $h\circ\gamma_\zeta([0,S_\zeta])$ through $h^{-1}$ would be a singleton. Having started with a boundary orbit $\gamma_\zeta([0,S_\zeta])$ satisfying our initial assumptions, we are led to a contradiction. This signifies that $h\circ\gamma_\zeta([0,S_\zeta])$ is necessarily isolated from the one side i.e. there exist some $\epsilon_1,\epsilon_2>0$ such that either $
	L^+=\{w\in\C:\textup{Re}h(\zeta)+\epsilon_1<\textup{Re}w<\textup{Re}h(\zeta)+S_\zeta,\textup{Im}h(\zeta)<\textup{Im}w<\textup{Im}h(\zeta)+\epsilon_2\}\subset\Omega
	$
	or $L^-=\{w\in\C:\textup{Re}h(\zeta)+\epsilon_1<\textup{Re}w<\textup{Re}h(\zeta)+S_\zeta,\textup{Im}h(\zeta)-\epsilon_2<\textup{Im}w<\textup{Im}h(\zeta)\}\subset\Omega.
	$
	 It is easy to see that both cannot be true, otherwise our boundary orbit would be isolated. Set $\Omega^+:=\{w\in\Omega:\textup{Im}w>\textup{Im}h(\zeta)\}$ and $\Omega^-=\{w\in\Omega:\textup{Im}w<\textup{Im}h(\zeta)\}$. If $L^+\subset\Omega$, then write $\Omega^*=\Omega^+$, whereas if $L^-\subset\Omega$, write $\Omega^*=\Omega^-$. Denote by $R(\Omega^*)$ the reflection of $\Omega^*$ with respect to the horizontal line that carries $h\circ\gamma_\zeta([0,S_\zeta])$ and set $\widetilde{\Omega}=\Omega^*\cup h\circ\gamma_\zeta((0,S_\zeta)) \cup R(\Omega^*)$. Clearly $\widetilde{\Omega}$ is simply connected. Let $\D^*=h^{-1}(\Omega^*)$ and set $R(\D^*)$ its reflection through the circular arc $\gamma_\zeta((0,S_\zeta))$. Let $\widetilde{\D}=\D^*\cup\gamma_\zeta((0,S_\zeta)) \cup R(\D^*)$. As we said before, $h$ may be extended holomorphically through $\gamma_\zeta((0,S_\zeta))$. Still denoting by $h$ its extended version, by Schwarz's Reflection Principle we see that $h$ maps $\widetilde{\D}$ conformally onto $\widetilde{\Omega}$.
\end{construction} 

Having constructed the necessary domains, we are now ready to state our final theorem regarding boundary orbits of the third type. The simply connected domains $\widetilde{\Omega}$ and $\widetilde{\D}$ will play the roles that $\Omega$ and $\D$ played, respectively, in Theorem \ref{backward}.

\begin{theo}\label{thm:third}
	Let $(\phi_t)$ be a non-elliptic semigroup in $\D$ and suppose that $\zeta\in\partial\D$ is such that $\gamma_\zeta$ is a boundary orbit of the third type. Then, $\gamma_\zeta$ is Lipschitz if and only if
	\begin{equation}\label{eq:third}
		\limsup\limits_{t\to S_\zeta}\frac{\lambda_{\widetilde{\Omega}}(h(\zeta)+t)}{\lambda_{\widetilde{\D}}(\gamma_\zeta(t))}<+\infty,
	\end{equation}
	where $\widetilde{\Omega},\widetilde{\D}$ are the domains constructed above.
\end{theo}
\begin{proof}
	Let $h$ and $G$ be the Koenigs function and infinitesimal generator of $(\phi_t)$, respectively. $G$ can also be holomorphically extended through $\gamma_\zeta((0,S_\zeta))$, so still denoting by $G$ the extension, we can consider $G$ as a holomorphic function on $\widetilde{\D}$ (the pre-image of $\widetilde{\Omega}$ through the extended $h^{-1}$). Since the identity $\partial\phi_t(z)/\partial t=G(\phi_t(z))$ continues to hold in $\widetilde{\D}$, we understand that $\gamma_\zeta$ is Lipschitz if and only if $\limsup_{t\to S_\zeta}|G(\gamma_\zeta(t))|<+\infty$. However, the identity $G=1/h'$ also holds in $\widetilde{\D}$. Therefore, by (\ref{hyperbolicinvariance}),
	$$|G(\gamma_\zeta)|=\frac{1}{|h'(\gamma_\zeta(t))|}=\frac{\lambda_{\widetilde{\Omega}}(h(\gamma_\zeta(t)))}{\lambda_{\widetilde{\D}}(\gamma_\zeta(t))},$$
	for all $t\in(0,S_\zeta)$ and the desired result follows.
\end{proof}

We understand that depending on the shape of the initial Koenigs domain $\Omega$, the limit in (\ref{eq:third}) can sometimes be finite and some others infinite. Of course, Theorem \ref{thm:third} lacks the clearly intuitive nature of Theorem \ref{backward}. Nonetheless, it still provides a necessary and sufficient condition that may be verified for certain domains.

\begin{rmk}
	There is an analogue of Theorem \ref{thm:third} for elliptic semigroups, as well. However, recall that if $(\phi_t)$ is elliptic with spectral value $\mu\in\C$, $\textup{Re}\mu>0$, then its Koenigs domain $\Omega$ is $\mu$-spirallike and the image of $\gamma_\zeta([0,S_\zeta])$ through the Koenigs function $h$ is a spiral segment. Through a logarithmic mapping, we may transform this spiral segment to a rectilinear segment and part of $\Omega$ to a new domain $\Omega'$. Then, we may reflect $\Omega'$ as before and reach a similar condition. We omit this result for the sake of avoiding repetition.
\end{rmk}

\begin{rmk}
	An analogue of Remark \ref{remark} holds for boundary orbits too. Firstly, since exceptional orbits converge to the Denjoy-Wolff, we know that $\lim_{t\to+\infty}G(\gamma_\zeta(t))=0$; see the proof of Theorem \ref{thm:exc-isol}. Therefore, $\gamma_\zeta$ cannot be bi-Lipschitz. In addition, all isolated orbits are not bi-Lipschitz since they are not even Lipschitz. Finally, a boundary orbit of the third type is bi-Lipschitz if and only if it is Lipschitz since the modulus of the infinitesimal generator is always bounded from below along it, but not necessarily bounded from above.
\end{rmk}

\section{\textbf{Semigroups in simply connected domains}}

Let $(\phi_t)$ be a semigroup in $\D$. Let $D\subsetneq\mathbb{C}$ be a simply connected domain and $f:\D\to D$ a Riemann mapping. Then, through a conjugation formula, we may define the mappings 
\begin{equation}\label{s1}
\phi_t^D=f\circ\phi_t\circ f^{-1}:D\to D, \;\;\;t\ge0.
\end{equation} 
It can be easily verified that in this way, we construct a semigroup $(\phi_t^D)$ in the simply connected domain $D$. We call this semigroup the $D$\textit{-version} of $(\phi_t)$. For $\zeta\in D$, we denote by $\gamma_\zeta^D$ its forward orbit and by $\tilde{\gamma}_\zeta^D$ its backward orbit.

Out of all the properties of the semigroup $(\phi_t)$, the most important ones are those shared by all its versions. The objective of this section is to find out whether the fact that all forward orbits are Lipschitz is such a property. Clearly, the Lipschitz condition is not conformally invariant in general. Nevertheless, it might be conformally invariant when restricting ourselves to orbits of semigroups.

Let $G:\D\to\mathbb{C}$ and $G^D:D\to\mathbb{C}$ be the infinitesimal generators of $(\phi_t)$ and $(\phi_t^D)$, respectively. By the definition of the infinitesimal generator, we know that
$$\frac{\partial\phi_t^D(\zeta)}{\partial t}=G^D(\phi_t^D(\zeta)), \quad t\ge0, \; \zeta\in D.$$
Following exactly the footsteps of the previous section, we are able to prove the following:
\begin{lm}\label{general}
	Let $(\phi_t^D)$ be a semigroup in $D$ with infinitesimal generator $G^D$. Fix $\zeta\in D$. The forward orbit $\gamma_\zeta^D$ (respectively, the backward orbit $\tilde{\gamma}_\zeta^D$) is Lipschitz if and only if $$\limsup_{t\to+\infty}|G^D(\gamma_\zeta^D(t))|<+\infty$$ $$(respectively, \;\;\;\;\limsup_{t\to T_\zeta}|G^D(\tilde{\gamma}_\zeta^D(t))|<+\infty).$$
\end{lm}

Even though Lemma \ref{general} gives us a characterization of the Lipschitz condition for general simply connected domains, we would like to have a better understanding of the situation. In particular, it would be useful to know if every forward orbit is Lipschitz as in semigroups of the unit disk, or if we can state a better result for special classes of backward orbits.

Using the conjugation formula (\ref{s1}) of $(\phi_t^D)$ and the chain rule, we deduce $G^D(\phi_t^D(\zeta))=f'(\phi_t(z))G(\phi_t(z))$, for all $t\ge0$ and all $\zeta\in D$, $z\in\D$ with $\zeta=f(z)$. In this way, we may write $G^D$ as the product of two holomorphic functions of $\D$. So, with a multiple use of Lemma \ref{pommerenkelemma} and denoting the Koenigs function of $(\phi_t)$ by $h$ and its Koenigs domain by $\Omega$, we have
\begin{eqnarray}\label{upperbound}
	|G^D(\phi_t^D(\zeta))|&=&|f'(\phi_t(z))|\cdot|G(\phi_t(z))|
\le\frac{4\delta_D(f(\phi_t(z)))}{1-|\phi_t(z)|^2}\cdot\frac{1-|\phi_t(z)|^2}{\delta_\Omega(h(\phi_t(z)))} \nonumber \\  
	&\le&4\frac{\delta_D(f(\phi_t(z)))}{\delta_\Omega(h(z))}.
\end{eqnarray}

However, the asymptotic behavior of the orbits of $(\phi_t^D)$ requires further attention. Let $\tau$ be the Denjoy-Wolff point of $(\phi_t)$. Clearly, if $\tau\in\D$, then $(\phi_t^D)$ has $f(\tau)\in D$ as its own Denjoy-Wolff point and remains elliptic. On the other hand, if $\tau\in\partial\D$, then Carath\'{e}odory's Theorem implies that $\tau$ corresponds through $f$ to a prime end $\xi$ of $D$. Therefore, all the orbits of $(\phi_t^D)$ escape to the boundary of $D$ and $(\phi_t^D)$ is still non-elliptic. Nevertheless, the impression of the prime end $\xi$ might be greater than a singleton; see \cite[Chapter 4]{book} for details on prime ends and their impressions. This does not allow us to formally speak of a Denjoy-Wolff point for $(\phi_t^D)$ since in certain cases its orbits might cluster on a continuum. For this reason, we may say that $\xi$ is the \textit{Denjoy-Wolff prime end} of $(\phi_t^D)$ and all its orbits converge to $\xi$ in the Carath\'{e}odory topology of $D$. For a recent discussion on conditions which certify the existence of an actual Denjoy-Wolff point (and not just a Denjoy-Wolff prime end) in arbitrary simply connected domains, we refer to \cite{Benini}. 

Keeping the note above in mind, we immediately deduce the following result:

\begin{prop}
	Let $D\subsetneq\mathbb{C}$ be a simply connected domain and $(\phi_t^D)$ be a semigroup in $D$. 
	\begin{enumerate}
		\item[\textup{(a)}] If $(\phi_t^D)$ is elliptic, then every forward orbit is Lipschitz.
		\item[\textup{(b)}] If $(\phi_t^D)$ is non-elliptic and the impression of the Denjoy-Wolff prime end of $(\phi_t^D)$ does not contain $\infty$, then every forward orbit is Lipschitz.
	\end{enumerate}
\end{prop}

\begin{proof}
	(a) Fix $\zeta\in D$. Since $(\phi_t^D)$ is elliptic, it has a Denjoy-Wolff point $\zeta_0\in D$ and $\lim_{t\to+\infty}\phi_t^D(\zeta)=\zeta_0$. By \eqref{upperbound} we obtain
	$$\limsup_{t\to+\infty}|G^D(\phi_t^D(\zeta))|\le4\dfrac{\delta_D(\zeta_0)}{\delta_\Omega(h(z))}<+\infty.$$
	Hence Lemma \ref{general} yields that $\gamma_\zeta^D$ is Lipschitz.
	
	(b) Fix $\zeta\in D$. This time, $(\phi_t^D)$ has a Denjoy-Wolff prime end $\xi$ with bounded impression. Arguing as in the previous case, from relation \eqref{upperbound} we infer that $\limsup_{t\to+\infty}|G^D(\phi_t^D(\zeta))|=0$ since $\gamma_\zeta^D$ clusters on a bounded set of points of $\partial D$. As a result, $\gamma_\zeta^D$ is Lipschitz by Lemma \ref{general}.
\end{proof}

\begin{rmk}
But what about non-elliptic semigroups whose Denjoy-Wolff prime ends have unbounded impression? Working as above, we may see that
\begin{equation}\label{lowerbound}
|G^D(\phi_t^D(\zeta))|\ge\frac{1}{4}\cdot\frac{\delta_D(f(\phi_t(z)))}{\delta_\Omega(h(\phi_t(z)))}.
\end{equation}
Thus we may extract the following:
\begin{enumerate}
	\item[(i)] If the distance $\delta_D(f(\phi_t(z)))$ remains bounded, then by (\ref{upperbound}), every forward orbit is still Lipschitz. For instance, this happens if $D$ is a strip or if $D$ is a half-plane and $(\phi_t^D)$ is parabolic of finite shift.
	\item[(ii)] Suppose now that for some  $z\in\D$, the quantity $\delta_D(f(\phi_t(z)))$ diverges, but $\delta_\Omega(h(\phi_t(z)))$ remains bounded. Then (\ref{lowerbound}) dictates that $\gamma_\zeta^D$ is not Lipschitz. For example, this happens when $D$ is a half-plane and $(\phi_t^D)$ is hyperbolic or parabolic of positive hyperbolic step and infinite shift.
\end{enumerate}
Therefore, both Lipschitz and non-Lipschitz forward orbits may arise in semigroups of general simply connected domains.

\end{rmk}

Similar results can be deduced when distinguishing cases with regard to the regularity of a backward orbit of $(\phi_t^D)$ (recall that regularity is defined through the hyperbolic distance which is conformally invariant) and to the point where it lands.

\section*{Acknowledgements}

We thank V. Vellis for the interesting conversation about Lipschitz curves and the Hayman-Wu Theorem. We also thank M. D. Contreras and L. Rodr\'{i}guez-Piazza for the helpful discussion on the properties of the orbits of holomorphic semigroups.

\end{document}